\documentclass[11pt]{amsart}
\usepackage{amssymb, amsmath}
\usepackage[active]{srcltx}
\usepackage{hyperref}

\newtheorem{thm}{Theorem}[section]
\newtheorem{lem}[thm]{Lemma}%
\newtheorem{prop}[thm]{Proposition}%
\theoremstyle{remark}
\newtheorem{remark}[thm]{Remark}
\theoremstyle{plain}

\numberwithin{equation}{section}

\def\EE{{\mathbb E}}

\def\new{\operatorname{new}}

\def\Tr{\operatorname{Tr}}
\def\Span{\operatorname{Span}}

\def\veca{{\text{\boldmath$a$}}}

\def\vece{{\text{\boldmath$e$}}}

\def\vecm{{\text{\boldmath$m$}}}

\def\vecu{{\text{\boldmath$u$}}}
\def\vecv{{\text{\boldmath$v$}}}

\def\vecw{{\text{\boldmath$w$}}}
\def\vecx{{\text{\boldmath$x$}}}

\def\vecy{{\text{\boldmath$y$}}}
\def\vecz{{\text{\boldmath$z$}}}

\def\vecbeta{{\text{\boldmath$\beta$}}}

\def\scrB{{\mathcal B}}

\def\scrM{{\mathcal M}}

\def\scrP{{\mathcal P}}

\def\fA{{\mathfrak A}}

\def\fM{{\mathfrak M}}

\def\fS{{\mathfrak S}}

\def\fV{{\mathfrak V}}
\def\fW{{\mathfrak W}}

\def\Re{\operatorname{Re}}

\def\diag{\operatorname{diag}}
\def\dim{\operatorname{dim}}

\def\GL{\operatorname{GL}}

\def\SL{\operatorname{SL}}

\def\Tr{\operatorname{Tr}}

\def\supp{\operatorname{supp}}

\def\vol{\operatorname{vol}}

\def\transp{{\hspace{-1pt}^{\mathsf{T}}}}   %

\def\Onder#1#2#3#4#5{#1 \setbox0=\hbox{$#1$}\setbox1=\hbox{$#2$}
       \dimen0=.5\wd0 \dimen1=\dimen0 \dimen2=\dp0 \dimen3=\dimen2
       \advance\dimen0 by .5\wd1 \advance\dimen0 by -#4
       \advance\dimen1 by -.5\wd1 \advance\dimen1 by -#4
       \advance\dimen2 by -#3 \advance\dimen2 by \ht1
       \advance\dimen2 by 0.3ex \advance\dimen3 by #5
        \kern-\dimen0\raisebox{-\dimen2}[0ex][\dimen3]{\box1}
       \kern\dimen1}

\newcommand{\tA}{A^{-\mathsf{T}}}  %
\newcommand{\dd}{\mathsf{d}}

\newcommand{\hF}{\widehat{F}}
\newcommand{\tB}{\widetilde{B}}

\newcommand{\tchi}{\widetilde{\chi}}

\newcommand{\Q}{\mathbb{Q}}
\newcommand{\C}{\mathbb{C}}
\newcommand{\R}{\mathbb{R}}
\newcommand{\Z}{\mathbb{Z}}

\renewcommand{\mod}{\text{ mod }}

\newcommand{\col}{\: : \:}

\newcommand{\bn}{\mathbf{0}}
\newcommand{\bs}{\backslash}

\newcommand{\tN}{\widetilde{N}}
\newcommand{\tn}{\widetilde{n}}

\newcommand{\trho}{{\widetilde{\rho}}}
\newcommand{\tsigma}{{\widetilde{\sigma}}}

\newcommand{\tscrM}{\widetilde{\scrM}}
\newcommand{\tg}{\tilde{g}}
\newcommand{\tV}{\tilde{V}}

\newcommand{\ve}{\varepsilon}

\newcommand{\matr}[4]{\left( \begin{matrix} #1 & #2 \\ #3 & #4 \end{matrix} \right) }
\newcommand{\cmatr}[2]{\left( \begin{matrix} #1 \\ #2 \end{matrix} \right) }
\newcommand{\smatr}[4]{\left( \begin{smallmatrix} #1 & #2 \\ #3 & #4 \end{smallmatrix} \right) }
\newcommand{\scmatr}[2]{\left( \begin{smallmatrix} #1 \\ #2 \end{smallmatrix} \right) }

\begin{document}
\title{On a mean value formula for multiple sums %
over a lattice and its dual}
\author{Andreas Str\"ombergsson}
\address{Department of Mathematics, Uppsala University, Box 480, SE-75106, Uppsala, Sweden}
\email{astrombe@math.uu.se}

\author{Anders S\"odergren}
\address{Department of Mathematical Sciences, Chalmers University of Technology and the University of
Gothenburg, SE-412 96 Gothenburg, Sweden}
\email{andesod@chalmers.se}

\begin{abstract}
We prove a generalized version of Rogers' mean value formula in the space $X_n$ of unimodular lattices in $\R^n$,
which gives the mean value of 
a multiple sum over a lattice $L$ and its dual $L^*$.
As an application,
we prove that for $L$ random with respect to the $\SL_n(\R)$-invariant probability measure, 
in the limit of large dimension $n$,
the volumes determined by the 
lengths of the non-zero vectors $\pm\vecx$ in $L$ 
on the one hand,
and the non-zero vectors $\pm\vecx'$ in $L^*$
on the other hand,
converge weakly %
to two independent Poisson processes on the 
positive real line, both with intensity $\frac12$.
\end{abstract}

\thanks{Str\"ombergsson was supported by 
the Swedish Research Council (grant 2016-03360)
and by the Knut and Alice Wallenberg Foundation. 
S\"odergren was supported by grants from the Swedish Research Council (grants 2016-03759 and 2021-04605).}

\maketitle
\tableofcontents

\section{Introduction}
\label{INTROsec}

Let $n\geq 2$, 
and set
\begin{align*}
X_n=G/\Gamma,
\qquad\text{with }
G=\SL_n(\R)\text{ and }\Gamma=\SL_n(\Z).
\end{align*}
Also let $\mu$ be the $G$-invariant probability measure on $X_n$.
As usual,
for any $g\in G$ 
we identify the point $g\Gamma$ in $G/\Gamma$ with the lattice $g\Z^n=\{g\vecv\col\vecv\in\Z^n\}$ in $\R^n$;
in this way $X_n$ becomes the space of all lattices of covolume one in $\R^n$.

In 1945, Siegel \cite{cS45a}
proved the following fundamental integration formula:
For any continuous, compactly supported function $\rho$ on $\R^n$, %
\begin{align}\label{SIEGELintformula}
\int_{X_n}\sum_{\vecv\in L}\rho(\vecv)\,d\mu(L)=\int_{\R^n}\rho(\vecx)\,d\vecx+\rho(\bn),
\end{align}
where $d\vecx$ denotes Lebesgue measure on $\R^n$.
The integrand in the left-hand side,
i.e.\ the function $L\mapsto\sum_{\vecv\in L}\rho(\vecv)$ on $X_n$,
is often called the Siegel transform of the function $\rho$.
Ten years later, Rogers \cite{cR55}\footnote{See also the papers
\cite{wS57ab}
and
\cite{aMcR58a}
for alternative and corrected proofs.}
proved a generalization of \eqref{SIEGELintformula}
which for any $1\leq k<n$ 
gives an explicit formula for 
the average of a $k$-fold %
sum over $L$, i.e.
\begin{align*}
\int_{X_n}\sum_{\vecv_1,\ldots,\vecv_k\in L}\rho(\vecv_1,\ldots,\vecv_k)\,d\mu(L)
\end{align*}
for any function $\rho$ on $(\R^n)^k$
subject to suitable conditions.
See Theorem \ref{ROGERSFORMULATHM} below
for a precise statement.
Rogers' mean value formula has found a number of applications over the years;
see, e.g.,
\cite{cR56},
\cite{wS58b},
\cite{pSaS2006},
\cite{aS2011o},
\cite{jAgM2009},
\cite{mBaG2018},
\cite{aSaS2016}.
In recent years many variants of Rogers' formula have been developed,
with several new applications;
see, e.g.,
\cite{dKsY2018},
\cite{aGdKsY2019},
\cite{jH2019},
\cite{mAaGsY2020},
\cite{mAaGjH2021}.

Our main goal in the present paper is to prove a generalization of
Rogers' mean value formula,
giving an explicit formula for the average of a multiple sum over $L$ 
as well as the \textit{dual} lattice $L^*$, i.e.
\begin{align*}
&\int_{X_n}
\sum_{\vecv_1,\ldots,\vecv_{k_1}\in L}\sum_{\vecw_1,\ldots,\vecw_{k_2}\in L^*}
\rho(\vecv_1,\ldots,\vecv_{k_1},\vecw_1,\ldots,\vecw_{k_2})\,d\mu(L).
\end{align*}
Recall that the dual of a (full-rank) lattice $L$ in $\R^n$ is defined by
\begin{align*}
L^*=\{\vecw\in\R^n\col \vecw\cdot\vecv\in\Z\hspace{7pt}\forall\vecv\in L\},
\end{align*}
where the dot %
denotes the standard scalar product on $\R^n$.
The covolume of $L^*$ equals the inverse of that of $L$: %
$\vol(\R^n/L^*)=\vol(\R^n/L)^{-1}$;
in particular $L^*\in X_n$ for any $L\in X_n$.
In fact, under our identification $X_n=G/\Gamma$, the map
$L\mapsto L^*$ on $X_n$ is given by $g\Gamma\mapsto g^{-\mathsf{T}}\Gamma$,
where $g^{-\mathsf{T}}:=(g^{-1})\transp=(g\transp)^{-1}$
(this notation will be used throughout the paper).
The study of inequalities relating invariants of $L$ and $L^*$
is a classical topic within the geometry of numbers; %
see, e.g., 
\cite{wB93}
and the references therein.

\vspace{5pt}

To prepare for the statement of our main result, we introduce some further notation.
Throughout the paper, we will identify vectors in $\R^n$ with $n\times 1$ column matrices.
More generally, for any $m\geq1$,
we will find it convenient to identify the space
$M_{n,m}(\R)$ of $n\times m$ matrices
with the space $(\R^n)^m$ of $m$-tuples of vectors in $\R^n$
(by listing the column vectors of any matrix $x\in M_{n,m}(\R)$ in order from left to right).
For $m\leq n$ and $x\in M_{n,m}(\R)$, we write
\begin{align}\label{dddef}
\dd(x):=\sqrt{\det(x\transp x)}.
\end{align}
Note that under our identification $M_{n,m}(\R)=(\R^n)^m$,
$\dd(x)$ equals the 
$m$-dimensional volume of the
parallelotope in $\R^n$ spanned by the vectors in $x$. %
In particular $\dd(x)>0$ if and only if $x$ has rank $m$.
We write
\begin{align}\label{Umdef}
U_m:=\{x\in M_{n,m}(\R)\col \dd(x)>0\};
\end{align}
this is a dense open subset of $M_{n,m}(\R)$.

Given any positive integers $m_1,m_2$ satisfying 
$m:=m_1+m_2<n$,
and given any $\beta\in M_{m_1,m_2}(\R)$, we set
\begin{align}\label{SbetaDEF2}
S(\beta):=%
\{\langle x,y\rangle\in U_{m_1}\times U_{m_2}\col x\transp y=\beta\}.
\end{align}
This is a closed, regular submanifold of
$U_{m_1}\times U_{m_2}$ of dimension
$mn-m_1m_2$,
since the map $\langle x,y\rangle\mapsto x\transp y$ from
$U_{m_1}\times U_{m_2}$ to $M_{m_1,m_2}(\R)$
has everywhere full rank $m_1m_2$.
We equip $S(\beta)$ with a natural Borel
measure $\eta_\beta$ as follows.
For any $x\in U_{m_1}$, we set
\begin{align}\label{SbetapDEF}
S(\beta)'_x:=\{y\in M_{n,m_2}(\R)\col x\transp y=\beta\}
\end{align}
and 
\begin{align}
S(\beta)_x:=S(\beta)_x'\cap U_{m_2}=\{y\in M_{n,m_2}(\R)\col \langle x,y\rangle\in S(\beta)\}.
\end{align}
Note that
$S(\beta)'_x$ is an affine linear subspace of $M_{n,m_2}(\R)$ of dimension
$m_2(n-m_1)$;
we equip $S(\beta)'_x$ with its structure as Euclidean subspace of
$M_{n,m_2}(\R)\cong\R^{nm_2}$ with its standard Euclidean structure,
and let 
$\eta_{\beta,x}$ be the corresponding $m_2(n-m_1)$-dimensional Lebesgue volume measure on
$S(\beta)_x'$.
Note that $S(\beta)_x'\setminus S(\beta)_x$ has measure zero with respect to $\eta_{\beta,x}$;
we write $\eta_{\beta,x}$ also for the restriction of $\eta_{\beta,x}$ to $S(\beta)_x$.
Finally we define, for any Borel set $E\subset S(\beta)$:
\begin{align}\label{NUbetaDEFnew}
\eta_\beta(E):=\int_{U_{m_1}}\int_{S(\beta)_x}\chi_E(x,y)\,d\eta_{\beta,x}(y)\,\frac{dx}{\dd(x)^{m_2}},
\end{align}
where $dx$ is the standard $nm_1$-dimensional Lebesgue measure on
$M_{n,m_1}(\R)$, restricted to $U_{m_1}$.

For any $B\in M_{k,m}(\R)$, we write
\begin{align*}
V_B:=B\R^m=\{B\vecv\col\vecv\in\R^m\}
\end{align*}
for the image of $B$ viewed as a linear map from $\R^m$ to $\R^k$.
Next, for any %
$k\geq m$,
we let $M_{k,m}(\Z)^*$  \label{MZstardef}
be the set of all $B\in M_{k,m}(\Z)$
satisfying $\dd(B)>0$ and 
$V_B\cap\Z^k=B\Z^m=\{B\vecv\col\vecv\in\Z^m\}$.
The group $\GL_m(\Z)$ acts on $M_{k,m}(\Z)^*$ by
multiplication from the right.
We fix %
a subset $A_{k,m}\subset M_{k,m}(\Z)^*$ 
containing exactly one representative from each \label{Akmspec}
orbit in $M_{k,m}(\Z)^*/\GL_m(\Z)$.
Then, under our identification of $M_{k,m}(\R)$ with $(\R^k)^m$,
$M_{k,m}(\Z)^*$ is the set of all $m$-tuples of vectors in $\Z^k$
which span a primitive $m$-dimensional sublattice of $\Z^k$
(a sublattice $L\subset\Z^k$ is said to be primitive if $\Span_{\R}(L)\cap\Z^k=L$).
Furthermore, two matrices $B,B'\in M_{k,m}(\Z)^*$ lie in the same $\GL_m(\Z)$-orbit
if and only if $V_B=V_{B'}$.
It follows that the map $B\mapsto V_B$ gives a bijection from $A_{k,m}$ onto the set of
\textit{rational} $m$-dimensional linear subspaces of $\R^k$
(a linear subspace $V\subset\R^k$ is said to be rational if 
$V=\Span_{\R}(V\cap\Z^k)$). %

For any $\beta\in M_{m_1,m_2}(\Z)$ we define a 'weight' $W(\beta)$ as follows:
Let $\fW_{m_1}$ be the set of matrices $A=(a_{ij})\in M_{m_1,m_1}(\Z)$ such that
$a_{ij}=0$ for all $1\leq j<i\leq m_1$,
$a_{ii}>0$ for all $1\leq i\leq m_1$,
and $0\leq a_{ij}<a_{jj}$ for all $1\leq i<j\leq m_1$.\label{fWdef}
Let $\fW_\beta$ be the following subset:
\begin{align}\label{fWbetadef}
\fW_\beta=\{A\in\fW_{m_1}\col A^{-\mathsf{T}}\beta\in M_{m_1,m_2}(\Z)\}.\hspace{10pt} %
\end{align}
Finally set:
\begin{align}\label{Wdef}
W(\beta):=\frac{\sum_{A\in \fW_{\beta}}(\det A)^{m_2-n}}{\zeta(n)\zeta(n-1)\cdots\zeta(n-m_1+1)}.
\end{align}

We are now ready to state the main result of the paper:
\pagebreak
\begin{thm}\label{MAINTHEOREMabs2}
Let $k_1,k_2\in\Z^+$, $n>k_1+k_2$,
and let $\rho:(\R^n)^{k_1}\times(\R^n)^{k_2}\to\R_{\geq0}$
be a non-negative Borel measurable function 
(which we identify with a function 
$M_{n,k_1}(\R)\times M_{n,k_2}(\R)\to\R_{\geq0}$).
Then
\begin{align}\label{MAINTHEOREMabs2res}
&\int_{X_n}
\sum_{\vecv_1,\ldots,\vecv_{k_1}\in L}\sum_{\vecw_1,\ldots,\vecw_{k_2}\in L^*}
\rho(\vecv_1,\ldots,\vecv_{k_1},\vecw_1,\ldots,\vecw_{k_2})\,d\mu(L)
\\\notag
&=\sum_{m_1=1}^{k_1}\sum_{B_1\in A_{k_1,m_1}}\sum_{m_2=1}^{k_2}\sum_{B_2\in A_{k_2,m_2}}
\sum_{\beta\in M_{m_1,m_2}\!(\Z)} W(\beta)
\int_{S(\beta)}\rho\bigl(xB_1^{\mathsf{T}},yB_2^{\mathsf{T}}\bigr)\,d\eta_\beta(x,y)
\\\notag
&\hspace{10pt}
+\sum_{m_1=1}^{k_1}\sum_{B_1\in A_{k_1,m_1}}\int_{M_{n,m_1}(\R)}\rho\bigl(xB_1^{\mathsf{T}},0\bigr)\,dx
+\sum_{m_2=1}^{k_2}\sum_{B_2\in A_{k_2,m_2}}\int_{M_{n,m_2}(\R)}\rho\bigl(0,yB_2^{\mathsf{T}}\bigr)\,dy
+\rho(0,0),
\end{align}
where both sides are finite whenever $\rho$ is bounded and has bounded support.
\end{thm}

\begin{remark}\label{emptymatricesREM}
By introducing certain natural conventions concerning 
``empty matrices of dimensions $a\times 0$ and $0\times a$'' (any $a\geq0$),
and agreeing that for any $k\geq1$, $A_{k,0}$ is the singleton set containing the 
empty matrix of dimension $k\times 0$, 
one may incorporate
the last line of \eqref{MAINTHEOREMabs2res} in the double sum,
so that the \textit{whole} of the right-hand side of \eqref{MAINTHEOREMabs2res}
can be expressed as  %
\begin{align}\label{MAINTHEOREMabs2resALT}
\sum_{m_1=0}^{k_1}\sum_{B_1\in A_{k_1,m_1}}\sum_{m_2=0}^{k_2}\sum_{B_2\in A_{k_2,m_2}}
\sum_{\beta\in M_{m_1,m_2}\!(\Z)} W(\beta)
\int_{S(\beta)}\rho\bigl(xB_1^{\mathsf{T}},yB_2^{\mathsf{T}}\bigr)\,d\eta_\beta(x,y).
\end{align}
\end{remark}

\begin{remark}
It follows from the statement of Theorem \ref{MAINTHEOREMabs2}
that the first sum in the right-hand side of \eqref{MAINTHEOREMabs2res} is independent of the
choice of the sets of representatives
$A_{k_1,m_1}$ and $A_{k_2,m_2}$.
This fact is also fairly easy to verify directly.
Indeed, note that it suffices to verify that
for any given $B_1\in M_{k_1,m_1}(\Z)^*$ and
$B_2\in M_{k_2,m_2}(\Z)^*$, 
the sum
\begin{align*}
\sum_{\beta\in M_{m_1,m_2}\!(\Z)} W(\beta)\int_{S(\beta)}\rho\bigl(xB_1^{\mathsf{T}},yB_2^{\mathsf{T}}\bigr)\,d\eta_\beta(x,y)
\end{align*}
remains the same when $B_1$ and $B_2$ are replaced by $B_ 1\gamma_1$ and $B_ 2\gamma_2$,
respectively, for any $\gamma_1\in\GL_{m_1}(\Z)$ and $\gamma_2\in\GL_{m_2}(\Z)$.
This claim follows from Lemma \ref{etabetatransfLEM} below,
together with the fact that
the map $\beta\mapsto \gamma_1\beta\gamma_2\transp $ is a bijection of
$M_{m_1,m_2}(\Z)$ onto itself
which preserves the weight $W$;
cf.\ Lemma \ref{Winvlem}.
\end{remark}

\begin{remark}\label{DUALsymmrem}
The map $L\mapsto L^*$ is a diffeomorphism of $X_n$ onto itself which preserves the measure $\mu$;
hence the left-hand side of the formula \eqref{MAINTHEOREMabs2res}
is %
invariant under replacing $\langle k_1,k_2,\rho\rangle$ by
$\langle k_2,k_1,\trho\,\rangle$ where $\trho(x,y):=\rho(y,x)$.
In the right-hand side of \eqref{MAINTHEOREMabs2res}
this invariance %
is directly visible via %
the following symmetry relations,
which we prove in Section \ref{betatranspsymmSEC} below:
For any $m_1,m_2\in\Z^+$ and $\beta\in M_{m_1,m_2}(\Z)$
we have $W(\beta\transp)=W(\beta)$
and $\eta_{\beta\transp}=J_*\eta_\beta$, where $J$ is the diffeomorphism
$\langle x,y\rangle\mapsto\langle y,x\rangle$ from $S(\beta)$ onto $S(\beta\transp)$.
\end{remark}

The formula in Theorem \ref{MAINTHEOREMabs2}
is also valid
if either $k_1=0$ or $k_2=0$, if  appropriately interpreted.
Indeed, in this case the formula becomes the following
slightly reformulated version of 
Rogers' original formula:
\begin{thm}\label{ROGERSFORMULATHMp}
(Rogers \cite{cR55}.)
Let $1\leq k<n$ and let $\rho:(\R^n)^k\to\R_{\geq0}$
be a non-negative Borel measurable function.
Then
\begin{align}\label{ROGERSFORMULATHMprimres}
\int_{X_n}\sum_{\vecv_1,\ldots,\vecv_k\in L}\rho(\vecv_1,\ldots,\vecv_k)\,d\mu(L)
=\sum_{m=1}^k\sum_{B\in A_{k,m}}\int_{M_{n,m}(\R)}\rho(xB\transp)\,dx
+\rho(0).
\end{align}
\end{thm}
(We discuss in Section \ref{ROGERSFORMULAsec} the translation between
the statement %
in \cite{cR55} and 
Theorem \ref{ROGERSFORMULATHMp}.)

\vspace{5pt}

Next, we discuss some applications of our main theorem.

\vspace{3pt}

Our main motivation for developing the formula in 
Theorem \ref{MAINTHEOREMabs2}
comes from questions concerning
the Epstein zeta function of a random lattice $L\in X_n$ as $n\to\infty$;
cf.\ \cite{pSaS2006, epstein1, epstein2, aSaS2016}.
Recall that for $\Re s>\frac n2$ and $L\in X_n$ the Epstein zeta function is defined by the absolutely convergent series
\begin{equation*}
E_n(L,s):=\sum_{\vecm\in L\setminus\{\bn\}}|\vecm|^{-2s}.
\end{equation*}  
The function $E_n(L,s)$ can be meromorphically continued to $\C$ and satisfies a functional equation of
"Riemann type" relating $E_n(L,s)$ and $E_n(L^*,\frac n2-s)$. 
An outstanding question from \cite{epstein2} is whether $E_n(L,s)$
for $s$ on or near the central point $s=\frac n4$, possesses,
after appropriate normalization, a limit distribution as $n\to\infty$?
It seems clear that Theorem \ref{MAINTHEOREMabs2}
should be an important ingredient when seeking to
extend the methods of
\cite{epstein2} to handle this question.
We hope to return to these matters in future work.

\vspace{5pt}

Finally, we describe an application of Theorem \ref{MAINTHEOREMabs2}
to the limit distribution of the shortest vector lengths of %
$L$ and $L^*$,
for $L$ random in $(X_n,\mu)$, as the dimension $n$ tends to infinity. %
Given a lattice $L\in X_n$, 
we order its non-zero vectors by increasing lengths as $\pm\vecv_1,\pm\vecv_2,\pm\vecv_3,\ldots$ and define, for each $j\geq1$,
\begin{equation*}
\mathcal V_j(L):=\fV_n|\vecv_j|^n,
\end{equation*} 
where $\fV_n$ denotes the volume of the unit ball in $\R^n$. 
In \cite{aS2011o}, the second author of the present paper
proved that for $L$ random in $(X_n,\mu)$,
the sequence $\{\mathcal V_j(L)\}_{j=1}^{\infty}$ converges in distribution, as $n\to\infty$, 
to the points of a Poisson process 
on $\R^+$ with constant intensity $\frac{1}{2}$.  
Of course the same limit result holds for the sequence
$\{\mathcal V_j(L^*)\}_{j=1}^{\infty}$,
since the map $L\mapsto L^*$ preserves the measure $\mu_n$.
The following result generalizes these facts by describing the \textit{joint} limiting distribution
of the normalized vector lengths of both $L$ and $L^*$.

\begin{thm}\label{JOINTPOISSONTHEOREM}
For $L$ random in $(X_n,\mu)$,
the two sequences $\{\mathcal V_j(L)\}_{j=1}^{\infty}$
and $\{\mathcal V_j(L^*)\}_{j=1}^{\infty}$ converge jointly in distribution, 
as $n\to\infty$, to the sequences $\{ T_j\}_{j=1}^{\infty}$
and $\{ T_j'\}_{j=1}^{\infty}$, where $0<T_1<T_2<T_3<\cdots$
and $0<T_1'<T_2'<T_3'<\cdots$ denote the points of two independent 
Poisson processes on $\R^+$ with constant intensity $\frac{1}{2}$.  
\end{thm}
\newpage

\section{Rogers' formula}
\label{ROGERSFORMULAsec}

In this section we explain the translation between Rogers' original formulation of
his mean value formula, \cite{cR55},
and the statement in Theorem \ref{ROGERSFORMULATHMp}.
The following is Rogers' original formulation:
\begin{thm}\label{ROGERSFORMULATHM}
Let $1\leq k<n$ and let $\rho:(\R^n)^k\to\R_{\geq0}$
be a non-negative Borel measurable function.
Then
\begin{align}\label{ROGERSFORMULATHMres}
\int_{X_n}\sum_{\vecv_1,\ldots,\vecv_k\in L}\rho(\vecv_1,\ldots,\vecv_k)\,d\mu(L)
=
\rho(0,\ldots,0)
\hspace{170pt}
\\\notag
+\sum_{D}
\Big(\frac{e_1}{q}\cdots\frac{e_m}{q}\Big)^n\int_{\R^n}\cdots\int_{\R^n}\rho\Big(\sum_{i=1}^m\frac{d_{i1}}{q}\vecx_i,\ldots,\sum_{i=1}^m\frac{d_{ik}}{q}\vecx_i\Big)\,d\vecx_1\ldots d\vecx_m,
\end{align}
where the inner sum is over all integer matrices $D=(d_{ij})\in M_{m,k}(\Z)$
with $m$ arbitrary in $\{1,\ldots,k\}$,
satisfying the following properties:
No column of $D$ vanishes identically;
the entries of $D$ have greatest common divisor equal to 1;
and finally there exists a division $(\nu;\mu)=(\nu_1,\ldots,\nu_m;\mu_1,\ldots,\mu_{k-m})$ 
of the numbers $1,\ldots,k$ into two sequences $\nu_1,\ldots,\nu_m$ and $\mu_1,\ldots,\mu_{k-m}$, satisfying
\begin{align}\label{Rdivision}
& 1\leq\nu_1<\nu_2<\ldots<\nu_m\leq k,\nonumber\\
& 1\leq\mu_1<\mu_2<\ldots<\mu_{k-m}\leq k,\\
& \nu_i\neq\mu_j, \text{ if $1\leq i\leq m$, $1\leq j \leq k-m$},\nonumber
\end{align}
such that, for some $q=q(D)\in\Z^+$,
\begin{align}\label{Rdivision2}
&d_{i\nu_j}=q\delta_{ij}, \hspace{10pt}i=1,\ldots,m,\,\,j=1,\ldots,m,\\
&d_{i\mu_j}=0,\hspace{10pt}\text{ if }\,\mu_j<\nu_i,\,\,i=1,\ldots,m,\,\,j=1,\ldots,k-m.\nonumber
\end{align}
Finally $e_i=\gcd(\ve_i,q)$, $i=1,\ldots,m$, where $\ve_1,\ldots,\ve_m$ are the elementary divisors of the matrix $D$. 
\end{thm}

\begin{remark}\label{ROGERSFORMULATHMabsconvrem}
Regarding convergence in \eqref{ROGERSFORMULATHMres},
Schmidt in \cite{wS58} proved that %
both sides of \eqref{ROGERSFORMULATHMres} are finite
whenever $\rho$ is bounded and has bounded support in $(\R^n)^k$.
In fact, Schmidt's proof %
applies verbatim to any non-negative $\rho$ of sufficiently rapid polynomial decay.
It therefore follows that \eqref{ROGERSFORMULATHMres} holds, with absolute convergence in both sides,
for all complex valued $\rho$ of sufficiently rapid polynomial decay,
and in particular for any 
Schwartz function $\rho$ on $(\R^n)^k$. %
\end{remark}

\begin{remark}
In the formula \eqref{ROGERSFORMULATHMres} in Theorem \ref{ROGERSFORMULATHM},
vector notation is used in both the left and the right-hand side.
However in the right-hand side of
\eqref{ROGERSFORMULATHMprimres} in Theorem \ref{ROGERSFORMULATHMp},
the identification $M_{n,m}(\R)=(\R^n)^m$ is used
to instead view $\rho$ as a function on $M_{n,m}(\R)$;
in particular the ``$0$'' in ``$\rho(0)$'' stands for the zero matrix in $M_{n,m}(\R)$.
It should also be noted that the right-hand side of 
\eqref{ROGERSFORMULATHMprimres}
is trivially independent of the choice of the set of representatives $A_{k,m}$.
Indeed, for any $B\in M_{n,m}(\Z)^*$ and any $\gamma\in\GL_m(\Z)$,
by substituting $x_{new}=x\gamma\transp$ we get
$\int_{M_{n,m}(\R)}\rho(xB\transp)\,dx
=\int_{M_{n,m}(\R)}\rho(x(B\gamma)\transp)\,dx.$
\end{remark}

As a preparation 
for the translation between Theorem \ref{ROGERSFORMULATHM} 
and Theorem \ref{ROGERSFORMULATHMp}
we first prove the following lemma,
which implies that in Theorem \ref{ROGERSFORMULATHM}
we have in fact $e_i=\ve_i$ for all $i$ and all $D$.

\begin{lem}\label{ELEMDIVBASIClem}
For any matrix $D$ as in Theorem \ref{ROGERSFORMULATHM},
with associated $q=q(D)$ and elementary divisors
$\ve_1,\ldots,\ve_m$,
we have $\ve_i\mid q$ for all $i$.
\end{lem}
\begin{proof}
Given $D$,
by the Smith normal form theorem,
there exist $\gamma_1\in\GL_m(\Z)$ and 
$\gamma_2\in\GL_k(\Z)$ such that
\begin{align}\label{ELEMDIVBASIClempf1}
D=\gamma_1\diag[\ve_1,\ldots,\ve_m]\bigl(I_m\: 0\bigr)\gamma_2,
\end{align}
where $\bigl(I_m\: 0\bigr)$ is block notation for an $m\times k$ matrix.
By considering only columns number $\nu_1,\ldots,\nu_m$ of the matrix $D$,
it follows that
\begin{align}\label{ELEMDIVBASIClempf2}
q\gamma_1^{-1}=\diag[\ve_1,\ldots,\ve_m] B,
\end{align}
where $B\in M_m(\Z)$ is the submatrix of $\gamma_2$ formed by removing the bottom $k-m$ rows and 
the columns of indices $\mu_1,\ldots,\mu_{k-m}$.
It follows from $\gamma_1^{-1}\in\GL_m(\Z)$ that for each $i\in\{1,\ldots,m\}$,
the entries in row number $i$ of the matrix
$q\gamma_1^{-1}$ have greatest common divisor $q$.
On the other hand, \eqref{ELEMDIVBASIClempf2} implies that
all these entries are divisible by $\ve_i$.
Hence $\ve_i\mid q$.
\end{proof}

\begin{proof}[Proof that Theorem \ref{ROGERSFORMULATHMp} is a reformulation of Theorem \ref{ROGERSFORMULATHM}]
For each $m\in\{1,\ldots,k\}$,
let $\fA(m)$ be the set of matrices of fixed size $m\times k$ appearing in the 
sum in \eqref{ROGERSFORMULATHMres}.
For each $D\in\fA(m)$,
we fix a choice of 
$\gamma_1=\gamma_1(D)\in\GL_m(\Z)$ and 
$\gamma_2=\gamma_2(D)\in\GL_k(\Z)$
such that \eqref{ELEMDIVBASIClempf1} holds.
It is easy to prove that the map
$D\mapsto D\transp\R^m$ 
is a bijection from $\fA(m)$ %
onto the family of non-zero rational linear subspaces of $\R^k$.
Note also that 
\eqref{ELEMDIVBASIClempf1} implies
$D\transp\R^m=\gamma_2' \R^m$,
where $\gamma_2'=\gamma_2'(D):=\bigl(\bigl(I_m\:0\bigr)\gamma_2\bigr)\transp$
(in other words, $\gamma_2'$ is the transpose of the top $m\times k$ submatrix of $\gamma_2$);
and since $\gamma_2\in\GL_k(\Z)$, we have $\gamma_2'\in M_{k,m}(\Z)^*$.
It follows that %
we may choose 
the set of representatives $A_{k,m}$ (cf.\ p.\ \pageref{Akmspec})
to be given by
\begin{align*}
A_{k,m}:=\{\gamma_2'(D)\col D\in\fA(m)\}.
\end{align*}

For each $D\in\fA(m)$,
in terms of our matrix notation,
the multiple integral appearing in the right-hand side of 
\eqref{ROGERSFORMULATHMres} equals
\begin{align}\label{BtoDkeyidentity}
\int_{M_{n,m}(\R)}\rho(x q^{-1}D)\,dx
=q^{nm}\int_{M_{n,m}(\R)}\rho(x D)\,dx
=\Bigl(\frac q{\ve_1}\cdots\frac q{\ve_m}\Bigr)^n\int_{M_{n,m}(\R)}\rho(y{\gamma_2'}\transp)\,dy,
\end{align}
where we used \eqref{ELEMDIVBASIClempf1}
and substituted $x=y\gamma_1\diag[\ve_1,\ldots,\ve_m]$.
Hence,
using also Lemma \ref{ELEMDIVBASIClem},
we conclude that the second line of \eqref{ROGERSFORMULATHMres}
equals
\begin{align*}
\sum_{m=1}^k\sum_{B\in A_{k,m}}
\int_{M_{n,m}(\R)}\rho(yB\transp)\,dy,
\end{align*}
and it follows that the right-hand side of \eqref{ROGERSFORMULATHMres}
equals the right-hand side of \eqref{ROGERSFORMULATHMprimres}.
\end{proof}

Finally, for later reference, we also state a formula which is more basic than
Theorem \ref{ROGERSFORMULATHM} and which can be used in the proof of
Theorem \ref{ROGERSFORMULATHM}.
For any lattice $L\in X_n$ and $1\leq k<n$, we write $P_{L,k}$ for the set of 
primitive $k$-tuples in $L^k$.
\begin{thm}\label{ROGERSbasicformulaTHM}
Let $1\leq k<n$ and let $\rho:(\R^n)^k\to\R_{\geq0}$
be a non-negative Borel measurable function.
Then
\begin{align}\label{ROGERSbasicformulaTHMres}
\int_{X_n}\sum_{\langle\vecv_1,\ldots,\vecv_k\rangle\in P_{L,k}}\rho(\vecv_1,\ldots,\vecv_k)\,d\mu(L)
=\frac1{\prod_{j=n-k+1}^n\zeta(j)} %
\int_{(\R^n)^k}\rho\, d\vecx_1\cdots d\vecx_k.
\end{align}
\end{thm}
\begin{proof}
See Macbeath and Rogers,
\cite[Theorem 1 and formula (13)]{aMcR58a}.
\end{proof}
\newpage

\section{Proof of the main theorem}
\label{DIRECTPROOFsec}

In this section we prove Theorem \ref*{MAINTHEOREMabs2}.

\subsection{Initial decompositions}

Recall that for any positive integer $m$,
we identify $(\R^n)^m$ with $M_{n,m}(\R)$.
In particular, %
for any lattice $L\in X_n$, the set $L^m$ is 
identified with a lattice in $M_{n,m}(\R)$,
and the left-hand side of \eqref{MAINTHEOREMabs2res}
can be expressed as follows:
\begin{align}\label{MAINTHMpf1}
&\int_{X_n}\sum_{x\in L^{k_1}}\sum_{y\in (L^*)^{k_2}}\rho(x,y)\,d\mu(L),
\end{align}
where $\rho$ is now a function on the space
\begin{align*}
\scrM=\scrM_{n,k_1,k_2}:=M_{n,k_1}(\R)\times M_{n,k_2}(\R).
\end{align*}
We will study the expression in \eqref{MAINTHMpf1} 
as a functional on the space $C_c(\scrM)$ of continuous functions $\rho$ on $\scrM$ of compact support.
A key observation is that if we let the group $G=\SL_n(\R)$ act on $\scrM$ in the following way:
\begin{align}\label{Gaction}
g\langle x,y\rangle:=\big\langle gx,g^{-\mathsf{T}}y\big\rangle,
\qquad \text{for any }g\in G,\: \langle x,y\rangle\in\scrM, %
\end{align}
then this functional is \textit{$G$-invariant:}
\begin{lem}\label{GinvfunctionalLEM}
The map $F:C_c(\scrM)\to\C$ given by
\begin{align}\label{GinvfunctionalLEMres}
F(\rho)=\int_{X_n}\sum_{x\in L^{k_1}}\sum_{y\in (L^*)^{k_2}}\rho(x,y)\,d\mu(L)
\end{align}
is a $G$-invariant positive linear functional on $C_c(\scrM)$. %
\end{lem}
\begin{proof}
We first prove that the functional $F$ is well-defined,
by verifying that the expression in \eqref{GinvfunctionalLEMres} is absolutely convergent
in the sense that 
$\int_{X_n}\sum_{x\in L^{k_1}}\sum_{y\in (L^*)^{k_2}}|\rho(x,y)|\,d\mu(L)<\infty$
for any $\rho\in C_c(\scrM)$.
To do so, it suffices to verify that
\begin{align}\label{GinvfunctionalLEMpf1}
\int_{X_n}\sum_{x\in L^{k_1}}\sum_{y\in (L^*)^{k_2}}\chi_K(x,y)\,d\mu(L)<\infty
\end{align}
for any compact subset $K$ of $\scrM$.
To this end, set $k=k_1+k_2$ and consider the Schwartz function
$\varphi(\vecx_1,\ldots,\vecx_k):=\prod_{j=1}^k e^{-\pi\|\vecx_j\|^2}$ on $\scrM$,
where we have identified $\scrM$ with $(\R^n)^k$
via our identifications $M_{n,k_1}(\R)=(\R^n)^{k_1}$.
It follows from Remark \ref{ROGERSFORMULATHMabsconvrem}
that 
\begin{align}\label{GinvfunctionalLEMpf2}
\int_{X_n}\sum_{x\in L^{k_1}}\sum_{y\in L^{k_2}}\varphi(x,y)\,d\mu(L)<\infty.
\end{align}
But for any fixed lattice $L\in X_n$ we have
$\sum_{\vecx\in L}e^{-\pi\|\vecx\|^2}
=\sum_{\vecx\in L^*}e^{-\pi\|\vecx\|^2}$,
by the Poisson summation formula,
and since the function $\vecx\mapsto e^{-\pi\|\vecx\|^2}$ ($\vecx\in\R^n$) is its own Fourier transform.
Hence also
$\sum_{x\in L^{k_1}}\sum_{y\in L^{k_2}}\varphi(x,y)
=\sum_{x\in L^{k_1}}\sum_{y\in (L^*)^{k_2}}\varphi(x,y)$ for every $L\in X_n$.
Using this fact in \eqref{GinvfunctionalLEMpf2},
together with the fact that 
$\inf_K\varphi>0$ for any given compact set $K\subset\scrM$,
we obtain \eqref{GinvfunctionalLEMpf1},
thus completing the proof of the absolute convergence in \eqref{GinvfunctionalLEMres}.

Finally we prove that $F$ is $G$-invariant.
Given $\rho\in C_c(\scrM)$ and $g\in G$, let us write
$\rho_g(x,y):=\rho(g\langle x,y\rangle)$.
For any $L\in X_n$ %
we have $g^{-\mathsf{T}}L^*=(gL)^*$, and hence
\begin{align*}
\sum_{x\in L^{k_1}}\sum_{y\in (L^*)^{k_2}}\rho_g(x,y)
=\sum_{x\in (gL)^{k_1}}\sum_{y\in ((g^{-\mathsf{T}}L)^*)^{k_2}}\rho(x,y)
=\sum_{x\in (gL)^{k_1}}\sum_{y\in ((gL)^*)^{k_2}}\rho(x,y).
\end{align*}
Using this formula in \eqref{GinvfunctionalLEMres},
and the fact that $\mu$ is $G$-invariant,
it follows that $F(\rho_g)=F(\rho)$,
which is %
the desired $G$-invariance of $F$.
\end{proof}

Our strategy will now be to combinatorially decompose the above functional 
as a sum of several ``smaller'' $G$-invariant terms,
which are in a certain sense more basic and which we will
be able to identify in an %
explicit way.
The first decomposition step is standard:
In the expression \eqref{MAINTHMpf1},
we restrict the summation ranges for $x$ and $y$ to the set of
\textit{linearly independent} tuples of vectors in $L$ and $L^*$. %
(This is in analogy with, e.g., the discussion in 
\cite[between (36) and (42)]{cR55}.)
For this we use the following basic fact.
\begin{lem}\label{rankmintmatricesbijLEM}
For any $1\leq m\leq k\leq n$,
the map $\langle\gamma,B\rangle\mapsto\gamma B\transp$
is a bijection from $(M_{n,m}(\Z)\cap U_m)\times A_{k,m}$ onto 
the set of matrices in $M_{n,k}(\Z)$ of rank $m$.
\end{lem}
\begin{proof}
Let $C$ be a matrix in $M_{n,k}(\Z)$ of rank $m$.
Let $V$ be the column space of $C\transp$; this is an
$m$-dimensional subspace of $\R^k$ spanned by integer vectors;
hence rational;
hence there exists a unique $B\in A_{k,m}$ 
such that $V=V_B=B\R^m$,
that is, $\R^n C=\R^m B\transp$.
It also follows from $B\in A_{k,m}$
that each vector
$\vecw$ in $\R^m B\transp$
equals $\vecv B\transp$ for a \textit{unique} vector $\vecv\in\R^m$,
with $\vecv\in\Z^m$ if and only if $\vecw\in\Z^k$.
Applying this fact to each row vector of $C$,
it follows that there exists a unique matrix
$\gamma\in M_{n,m}(\Z)$ such that $C=\gamma B\transp$,
and this $\gamma$ must have full rank $m$
since $C$ has full rank $m$;
in other words $\gamma\in M_{n,m}(\Z)\cap U_m$.
\end{proof}

For any $g\in G$,
letting $L:=g\Z^n$ in $X_n$,
we have $L^k=(g\Z^n)^k=g\cdot M_{n,k}(\Z)$.
Hence it follows from Lemma \ref{rankmintmatricesbijLEM} 
that the map $\langle x,B\rangle\mapsto xB\transp$
is a bijection from $(L^m\cap U_m)\times A_{k,m}$ onto the set of 
tuples in $L^k$ whose linear span in $\R^n$ has dimension $m$.
Applying this to both the sums in \eqref{MAINTHMpf1},
it follows that \eqref{MAINTHMpf1} can be rewritten as
\begin{align}\label{MAINTHMpf10}
\sum_{m_1=1}^{k_1}\sum_{B_1\in A_{k_1,m_1}}\sum_{m_2=1}^{k_2}\sum_{B_2\in A_{k_2,m_2}}
\int_{X_n}\sum_{x\in L^{m_1}\cap U_{m_1}}\sum_{y\in (L^*)^{m_2}\cap U_{m_2}}\rho(xB_1^{\mathsf{T}},yB_2^{\mathsf{T}})\,d\mu(L)
\\\notag
+\int_{X_n}\sum_{x\in L^{k_1}}\rho(x,0)\,d\mu(L)
+\int_{X_n}\sum_{y\in (L^*)^{k_2}}\rho(0,y)\,d\mu(L)
-\rho(0,0).
\end{align}
By applying Theorem \ref{ROGERSFORMULATHMp}
to the two integrals in the last line
(where for the last integral we use the fact that the map $L\mapsto L^*$ is a diffeomorphism of $X_n$
preserving the measure $\mu$),
we find that the last line in \eqref{MAINTHMpf10}
exactly matches the last line in \eqref{MAINTHEOREMabs2res}.
Hence it remains to handle the first line in \eqref{MAINTHMpf10}.
Clearly, it will suffice to find an appropriate explicit expression for the following integral,
for any given $m_1,m_2\in\Z^+$ with $m_1+m_2<n$ and any given 
$\rho\in C_c(\scrM)$ (where now $\scrM=\scrM_{n,m_1,m_2}$):
\begin{align}\label{MAINTHMpf4}
&\int_{X_n}\sum_{x\in L^{m_1}\cap U_{m_1}}\sum_{y\in (L^*)^{m_2}\cap U_{m_2}}\rho(x,y)\,d\mu(L).
\end{align}
Note that the proof of Lemma \ref{GinvfunctionalLEM} carries over immediately
to show that also the expression in \eqref{MAINTHMpf4}
is a $G$-invariant positive linear functional on $C_c(\scrM)$.

The next decomposition step is to note that 
for each $\beta\in M_{m_1,m_2}(\R)$,
the $G$-action in \eqref{Gaction} 
\textit{preserves the submanifold $S(\beta)$} of 
$\scrM$. \footnote{Note that $S(\beta)\subset U_{m_1}\times U_{m_2}\subset\scrM$.}
We also note that for any $L\in X_n$
and any $x\in L^{m_1}$ and $y\in (L^*)^{m_2}$ we have $x\transp y\in M_{m_1,m_2}(\Z)$,
i.e.\ the pair $\langle x,y\rangle$ belongs to $S(\beta)$ for some \textit{integer} matrix $\beta\in M_{m_1,m_2}(\Z)$.
It is therefore natural to decompose 
\eqref{MAINTHMpf4} according to the values of $\beta=x\transp y$,
i.e., to rewrite \eqref{MAINTHMpf4} as
the sum over all $\beta\in M_{m_1,m_2}(\Z)$ of the following expression:
\begin{align}\label{MAINTHMpf5}
\int_{X_n}\sum_{x\in L^{m_1}\cap U_{m_1}}
\sum_{\substack{y\in (L^*)^{m_2}\cap U_{m_2} %
\\ (x\transp y=\beta)}} \rho(x,y)\,d\mu(L).
\end{align}
It turns out that for each fixed $\beta$ the corresponding term
in \eqref{MAINTHMpf5} is a $G$-invariant 
positive linear functional on $C_c(S(\beta))$
(this follows from \eqref{MAINTHMpf7} and Lemma \ref{GinvfunctionalLEM2} below).
We will also prove that the action of $G$ on $S(\beta)$ is
transitive, %
and therefore
such a $G$-invariant functional is
\textit{unique} up to scalar multiplication.
This will allow us to identify the functional in 
\eqref{MAINTHMpf5} in a completely explicit way.

The identification just mentioned will go via
one last decomposition step:
In \eqref{MAINTHMpf5},
we restrict the summation over $x$
to \textit{primitive} $m_1$-tuples in $L^{m_1}$,
i.e.\ tuples which can be extended to a $\Z$-basis of $L$.
That is, we consider the following expression:
\begin{align}\label{MAINTHMpf6}
F_\beta(\rho):=\int_{X_n}\sum_{x\in P_{L,m_1}}
\sum_{\substack{y\in (L^*)^{m_2}\cap U_{m_2} %
\\ (x\transp y=\beta)}}
\rho(x,y)\,d\mu(L),
\end{align}
where we recall that $P_{L,m_1}\subset L^{m_1}\cap U_{m_1}$
denotes the subset of primitive $m_1$-tuples in $L^{m_1}$.
The proof that \eqref{MAINTHMpf5} can be expressed in terms of \eqref{MAINTHMpf6}
goes via the following (standard) lemma.

\begin{lem}\label{PLmWmbijectionLEM}
For any lattice $L\subset\R^n$ and any $1\leq m<n$,
the map $\langle x,A\rangle\mapsto xA$
is a bijection from $P_{L,m}\times \fW_{m}$ onto $L^m\cap U_m$.
\end{lem}
(Here $\fW_m$ is the set introduced just above \eqref{fWbetadef}, with $m=m_1$.)
\begin{proof}
Immediate from \cite[Lemma 9]{aMcR58a}.
\end{proof}

It follows from Lemma \ref{PLmWmbijectionLEM} that
for any $\beta\in M_{m_1,m_2}(\Z)$ and $\rho\in C_c(S(\beta))$,
the expression in \eqref{MAINTHMpf5}
can be rewritten as %
\begin{align}\label{MAINTHMpf7}
&\int_{X_n}\sum_{A\in\fW_{m_1}}\sum_{x\in P_{L,m_1}}
\sum_{\substack{y\in (L^*)^{m_2}\cap U_{m_2} %
\\ ((xA)\transp y=\beta)}}
\rho(xA,y)\,d\mu(L)
=\sum_{A\in\fW_{\beta}}F_{\tA\beta}(\rho_A),
\end{align}
where $\rho_A\in C_c(S(\tA\beta))$ is defined by
$\rho_A(x,y)=\rho(xA,y)$,
and $F_\beta$ is given by \eqref{MAINTHMpf6};
recall also that $\fW_{\beta}$ denotes the set of all $A\in\fW_{m_1}$ satisfying 
$\tA\beta\in M_{m_1,m_2}(\Z)$.
The change of order of summation in \eqref{MAINTHMpf7}
is justified by absolute convergence;
indeed the %
left-hand side of
\eqref{MAINTHMpf7}, with $|\rho|$ in the place of $\rho$,
is majorized by the expression in \eqref{GinvfunctionalLEMpf1}
with $K=\supp(\rho)$.

\subsection{$G$-invariant measures on $S(\beta)$} %

\begin{lem}\label{GinvfunctionalLEM2}
For any $\beta\in M_{m_1,m_2}(\Z)$,
the function $F_\beta$ defined in \eqref{MAINTHMpf6}
is a $G$-invariant positive linear functional on $C_c(S(\beta))$.
\end{lem}
\begin{proof}
For any $\rho\in C_c(S(\beta))$, the expression giving $F_\beta(|\rho|)$
is majorized by the expression in \eqref{GinvfunctionalLEMpf1}
with $K=\supp(\rho)$;
hence $F_\beta$ is a well-defined positive linear functional on $C_c(S(\beta))$.
The $G$-invariance is verified by the same type of computation as in the proof of Lemma~\ref{GinvfunctionalLEM},
the main point being that for any $g\in G$ and $L\in X_n$ we have
\begin{align*}
\sum_{x\in P_{L,m_1}}
\sum_{\substack{y\in (L^*)^{m_2}\cap U_{m_2} %
\\ (x\transp y=\beta)}}
\rho\bigl(gx,g^{-\mathsf{T}}y\bigr)
=\sum_{x\in P_{gL,m_1}}\,
\sum_{\substack{y\in ((gL)^*)^{m_2}\cap U_{m_2} %
\\ (x\transp y=\beta)}}
\rho(x,y).
\end{align*}
\end{proof}

Next we will prove that also 
the measure $\eta_\beta$ introduced in \eqref{NUbetaDEFnew}
is $G$-invariant.
For this we will require the following simple auxiliary lemma.
We equip $\R^n$ with its standard Euclidean structure;
then any non-zero linear subspace $V$ of $\R^n$ inherits a structure 
as a Euclidean subspace, and we denote by $\vol_V$ the corresponding
Lebesgue volume measure on $V$.
For any $g\in\GL_n(\R)$ we let 
$\delta(g,V)>0$ be the volume scaling factor of 
the linear map $g\big|_V:\:\vecv\mapsto g\vecv$
from $V$ onto $gV$,
i.e.\ the number $\delta$ such that 
$\vol_{gV}(gE)=\delta\cdot\vol_V(E)$ for any Borel set $E\subset V$.
\begin{lem}\label{JACOBIANauxlem}
For any $g\in G$ and any non-trivial subspace $V\subset\R^n$,
$\delta(g\transp,(gV)^\perp)=\delta(g,V)^{-1}$.
\end{lem}
\begin{proof}
Note that $\delta(k,W)=1$ for any $k\in O(n)$ and any non-zero subspace $W\subset\R^n$.
Using this fact one verifies that
for any $k_1,k_2\in O(n)$
we have $\delta(g,V)=\delta(k_1gk_2,k_2^{-1}V)$
and $\delta(g\transp,(gV)^\perp)=\delta(k_2^{-1}g\transp k_1^{-1},k_1(gV)^\perp)$,
and here $k_2^{-1}g\transp k_1^{-1}={\tg}\transp$ 
and $k_1(gV)^\perp:=(\tg\tV)^\perp$
with $\tg:=k_1gk_2$ and $\tV:=k_2^{-1}V$.
Hence it suffices to prove the lemma with $\tg$ and $\tV$ in the place of $g$ and $V$.
By choosing $k_1,k_2$ appropriately, we may assume 
that $\tV$ equals the span of the first $d=\dim V$ standard basis vectors of $\R^n$
and also $\tg\tV=\tV$;
this means that in block matrix notation we have $\tg=\smatr{\alpha}{\beta}0{\gamma}$
for some $\alpha\in\GL_d(\R)$, $\beta\in M_{d,n-d}(\R)$, $\gamma\in\GL_{n-d}(\R)$.
Now $\delta(\tg,\tV)=|\det\alpha|$
and $\delta(\tg\transp,(\tg\tV)^\perp)=|\det(\gamma\transp)|=|\det\gamma|$,
and the lemma follows from the fact that $\det\alpha\det\gamma=\det\tg=\det g=1$.
\end{proof}

\begin{lem}\label{etabetaGinvLEM}
For any $\beta\in M_{m_1,m_2}(\R)$,
the measure $\eta_\beta$ on $S(\beta)$ is $G$-invariant.
\end{lem}
\begin{proof}
Let $g\in G$ and let $E$ be a Borel subset of $S(\beta)$.
Then by \eqref{NUbetaDEFnew},
\begin{align}\label{etabetaGinvLEMpf10}
\eta_\beta(g^{-1}E)=\int_{U_{m_1}}\int_{S(\beta)_x}\chi_E\bigl(gx,(g\transp)^{-1}y\bigr)\,d\eta_{\beta,x}(y)\,\frac{dx}{\dd(x)^{m_2}}.
\end{align}
To rewrite this expression,
consider any fixed $x\in U_{m_1}$, and set $V=x\R^{m_1}\subset\R^n$.
Then the map $z\mapsto g\transp z$ 
is a diffeomorphism
from $S(\beta)_{gx}$ onto $S(\beta)_x$,
and since $S(\beta)_{gx}$ and $S(\beta)_x$
are translates of the subspaces $((gV)^\perp)^{m_2}$ and $(V^\perp)^{m_2}$,
respectively,
inside $M_{n,m_2}(\R)=(\R^n)^{m_2}$,
we have for any Borel subset $E\subset S(\beta)_{gx}$:
\begin{align*}
\eta_{\beta,x}(g\transp E)=\delta(g\transp,(gV)^\perp)^{m_2}\eta_{\beta,gx}(E)
=\delta(g,V)^{-m_2}\eta_{\beta,gx}(E),
\end{align*}
where the last equality holds by Lemma \ref{JACOBIANauxlem}.
Hence
\begin{align}\label{etabetaGinvLEMpf11}
\int_{S(\beta)_x}\chi_E\bigl(gx,(g\transp)^{-1}y\bigr)\,d\eta_{\beta,x}(y)
=\delta(g,V)^{-m_2}\int_{S(\beta)_{gx}}\chi_E\bigl(gx,z\bigr)\,d\eta_{\beta,gx}(z).
\end{align}
Using this formula in \eqref{etabetaGinvLEMpf11},
together with the fact that $\delta(g,V)\,\dd(x)=\dd(gx)$,
we conclude:
\begin{align*}
\eta_\beta(g^{-1}E)=\int_{U_{m_1}}\int_{S(\beta)_{gx}}\chi_E\bigl(gx,z\bigr)\,d\eta_{\beta,gx}(z)\,\frac{dx}{\dd(gx)^{m_2}}
=\eta_\beta(E),
\end{align*}
where the last equality follows by substiting $x_{\new}:=gx$
and comparing with \eqref{NUbetaDEFnew}.
\end{proof}

Next our goal is to show that \textit{every} $G$-invariant regular Borel measure
on $S(\beta)$ is a constant multiple of $\eta_\beta$.
As we explain below, this is a standard consequence of %
the following lemma.
\begin{lem}\label{GtransitiveonSbetaLEM}
For any $\beta\in M_{m_1,m_2}(\R)$,
the action of $G$ on $S(\beta)$ is transitive.
\end{lem}
\begin{proof}
We will use block matrix notation.
Fix an arbitrary matrix $w_1$ in $M_{n-m_1,m_2}(\R)$ of full rank,
that is, whose row space equals $\R^{m_2}$,
and set $z=\cmatr{I_{m_1}}{0}\in M_{n,m_1}(\R)$ 
and $w=\cmatr{\beta}{w_1}\in M_{n,m_2}(\R)$.
Then $\langle z,w\rangle\in S(\beta)$.
It now suffices to prove that given an 
arbitrary element $\langle x,y\rangle\in S(\beta)$,
there exists $g\in G$ such that $g\langle x,y\rangle=\langle z,w\rangle$.
But the row space of $x$ equals $\R^{m_1}$
since $x\in U_{m_1}$;
hence there exists
$g\in G$ such that $gx=z$.
Hence we may just as well assume $x=z$
from the start.
It then follows from $\langle x,y\rangle\in S(\beta)$ that 
$y$ has the block decomposition $y=\cmatr{\beta}{y_2}$
for some $y_2\in M_{n-m_1,m_2}(\R)$.
Let $V\subset\R^{m_2}$ be the row space of $\beta$;
then $y\in U_{m_2}$ means that the row vectors of $y_2$ together with $V$ span all $\R^{m_2}$.
Using also $n-m_1>m_2$, it follows that there is a way to add
appropriate vectors from $V$ to the rows of $y_2$ so as to obtain a matrix with row space $\R^{m_2}$;
in other words, there exists some 
$a\in M_{n-m_1,m_1}(\R)$ such that 
the row space of $y_2+a\beta$ equals $\R^{m_2}$.
Then there is some $b\in\SL_{n-m_1}(\R)$ such that
$y_2+a\beta=bw_1$.
Setting now $g=\matr{I_{m_1}}{-a\transp}0{b\transp}\in G$,
we have $gz=z$ and $g\transp w=y$
and hence $g\langle x,y\rangle=\langle z,w\rangle$.
\end{proof}
It follows from Lemma \ref{GtransitiveonSbetaLEM} that for any $p\in S(\beta)$,
if $H=H_{\beta,p}=\{g\in G\col gp=p\}$
is the isotropy group at $p$,
then the map $gH\mapsto gp$ 
is a diffeomorphism from $G/H$ onto $S(\beta)$
\cite[Theorem 3.62]{fW83}.
Thus $F_\beta$ can be viewed as a $G$-invariant positive linear functional on
$C_c(G/H)$,
or equivalently a $G$-invariant regular Borel measure on $G/H$.
Now by a basic result on invariant measures on homogeneous spaces
(see e.g.\ \cite[Lemma 1.4]{mR72}),
it follows that $F_\beta$ is the \textit{unique} $G$-invariant
regular Borel measure %
on $S(\beta)$, up to a scalar multiple.
However, by Lemma \ref{etabetaGinvLEM},
the same is true for $\eta_\beta$;
hence we conclude that 
$F_\beta$ is a constant multiple of $\eta_\beta$,
for every $\beta\in M_{m_1,m_2}(\Z)$.

The next proposition
gives the explicit formula for the constant of proportionality;
in fact it turns out to be indendent of $\beta$:
\begin{prop}\label{cbetaexplPROP}
For any $\beta\in M_{m_1,m_2}(\Z)$
and any $\rho\in C_c(S(\beta))$,
\begin{align*}
F_\beta(\rho)=\frac1{\zeta(n)\zeta(n-1)\cdots\zeta(n-m_1+1)}\,\eta_\beta(\rho).
\end{align*}
\end{prop}
We will prove Proposition \ref{cbetaexplPROP} in Section \ref{cbetaexplPROPpfsec}.
Here we show how the proof of Theorem \ref{MAINTHEOREMabs2}
is completed using Proposition \ref{cbetaexplPROP}.
Let us first note a general transformation formula for the measure $\eta_\beta$.
\begin{lem}\label{etabetatransfLEM}
For any $T_1\in\GL_{m_1}(\R)$, $T_2\in\GL_{m_2}(\R)$ and $\beta\in M_{m_1,m_2}(\R)$,
the map $\langle x,y\rangle\mapsto \langle xT_1,yT_2\rangle$ is a diffeomorphism of $S(\beta)$ onto
$S\bigl(T_1^{\mathsf{T}}\beta T_2\bigr)$,
and for any $f\in C_c\bigl(S(T_1^{\mathsf{T}}\beta T_2)\bigr)$:
\begin{align}\label{etabetatransfLEMres}
\int_{S(\beta)}f\bigl(xT_1,yT_2\bigr)\,d\eta_\beta(x,y)
=|\det T_1|^{m_2-n}|\det T_2|^{m_1-n}\int_{S(T_1^{\mathsf{T}}\beta T_2)}f\,d\eta_{T_1^{\mathsf{T}}\beta T_2}.
\end{align}
\end{lem}
\begin{proof}
Immediate from the definition of $\eta_\beta$, \eqref{NUbetaDEFnew}.
\end{proof}

It follows from Proposition~\ref{cbetaexplPROP}
and the formula in \eqref{MAINTHMpf7},
that for any $\beta\in M_{m_1,m_2}(\Z)$ and $\rho\in C_c(S(\beta))$, 
the expression in \eqref{MAINTHMpf5} equals
\begin{align}\label{MAINTHMpf8}
\frac1{\zeta(n)\zeta(n-1)\cdots\zeta(n-m_1+1)}\sum_{A\in\fW_{\beta}}\int_{S(\tA\beta)}\rho(xA,y)\,d\eta_{\tA\beta}(x,y).
\end{align}
Here by Lemma \ref{etabetatransfLEM},
the integral over $S(\tA\beta)$ equals
$(\det A)^{m_2-n}\eta_\beta(\rho)$.
Recalling also \eqref{Wdef},
it follows that \eqref{MAINTHMpf8},
and hence also \eqref{MAINTHMpf5}, equals
$W(\beta)\,\eta_\beta(\rho)$.
We have also noted that the expression in
\eqref{MAINTHMpf4} equals
\eqref{MAINTHMpf5} added over all $\beta\in M_{m_1,m_2}(\Z)$.
Hence we conclude: %
\begin{align}\label{MAINTHMpf9}
&\int_{X_n}\sum_{x\in L^{m_1}\cap U_{m_1}}\sum_{y\in (L^*)^{m_2}\cap U_{m_2}}\rho(x,y)\,d\mu(L)
=\sum_{\beta\in M_{m_1,m_2}(\Z)}W(\beta)\,\eta_\beta(\rho).
\end{align}

Finally, using the formula 
\eqref{MAINTHMpf9}
to evaluate the integral appearing in the first line of
\eqref{MAINTHMpf10},
we finally obtain that the expression in 
\eqref{MAINTHMpf10}
equals the right-hand side of \eqref{MAINTHEOREMabs2res};
in other words the equality in \eqref{MAINTHEOREMabs2res} holds.
We have proved this for any $\rho\in C_c(\scrM)$,
and it is clear from the proof that all the iterated sums and integrals
in \eqref{MAINTHEOREMabs2res} are absolutely convergent for any such $\rho$.
Hence the equality in \eqref{MAINTHEOREMabs2res} can be viewed as
an equality between two regular Borel measures on $\scrM$,
integrated against the function $\rho$;
and it is then clear that the equality extends to the 
class of test functions
$\rho$ appearing in the statement of Theorem \ref{MAINTHEOREMabs2}.
\hfill$\square$ %

\subsection{Proof of the explicit formula for $F_\beta$}
\label{cbetaexplPROPpfsec}
In this section we prove Proposition \ref{cbetaexplPROP}.
As we have shown, this will also complete the proof of Theorem \ref{MAINTHEOREMabs2}.

We start by giving three %
lemmas.
For $L$ any lattice in $\R^n$ (of full rank, but not necessarily of covolume one)
we write $\lambda_i=\lambda_i(L)$ ($i=1,\ldots,n$)
for its successive minima with respect to the unit ball,
i.e.
\begin{align*}
\lambda_i(L):=\min\bigl\{\lambda\in\R_{\geq0}\col L\text{ contains $i$ linearly independent vectors of length $\leq\lambda$}\bigr\}.
\end{align*}
\begin{lem}\label{LATTICECOUNTBOUNDlem1}
For any lattice in $L$ in $\R^n$ and any $R>0$,
\begin{align*}
\#(L\cap\scrB_R^n)\asymp_n\prod_{i=1}^n\Bigl(1+\frac R{\lambda_i(L)}\Bigr).
\end{align*}
\end{lem}
\begin{proof}
See \cite[Proposition 6]{hGcS91}.\footnote{Note that there is an obvious misprint in the
statement of this proposition; ``$\lambda_1\cdots\lambda_k/M(K)$'' should read
``$\lambda_1\cdots\lambda_k M(K)$''; similarly, in the last line of the proof,
``$\sim V(K_0)$'' should be corrected to ``$\sim V(K_0)^{-1}$''.}
\end{proof}

\begin{lem}\label{LATTICECOUNTBOUNDlem2}
For any lattice $L$ in $\R^n$ and any $\vecv\in\R^n$ and $0<R_1\leq R_2$,
\begin{align*}
\#(L\cap(\scrB_{R_2}^n+\vecv))\ll_n \Bigl(\frac{R_2}{R_1}\Bigr)^n\cdot\#(L\cap\scrB_{R_1}^n).
\end{align*}
\end{lem}
\begin{proof}
See \cite[Propositions 4 and 5]{hGcS91}.
\end{proof}

Let us write $\fV_m=\frac{\pi^{\frac12m}}{\Gamma(\frac12m+1)}$, the volume of the unit ball in $\R^m$.
\begin{lem}\label{R1R2limitLEM1}
For any fixed $\beta\in M_{m_1,m_2}(\Z)$, $L\in X_n$ and $x\in P_{L,m_1}$,
we have
\begin{align}\label{R1R2limitLEM1res}
\lim_{R\to\infty} R^{-(n-m_1)m_2}\#\bigl((L^*)^{m_2}\cap S(\beta)_x\cap(\scrB_{R}^n)^{m_2}\bigr)
=\frac{\fV_{n-m_1}^{m_2}}{\dd(x)^{m_2}}.
\end{align}
\end{lem}
\begin{proof}
Set $V=x\R^{m_1}\subset\R^n$.
Since the columns of $x$ form a primitive $m_1$-tuple in $L$,
$x$ can be extended to a matrix $g\in G$ with block decomposition $g=(x\hspace{5pt}*)$
such that $L=g\Z^n$.
Set $\beta':=\scmatr{\beta}0\in M_{n,m_2}(\Z)$
and $y:=g^{-\mathsf{T}}\beta'$;
then $y\in (L^*)^{m_2}\cap S(\beta)_x'$.
It follows that $S(\beta)_x'=y+(V^\perp)^{m_2}$,
and so 
\begin{align}\notag
\#\bigl((L^*)^{m_2}\cap S(\beta)_x'\cap(\scrB_{R}^n)^{m_2}\bigr)
&=\#\bigl((y+(L^*\cap V^\perp)^{m_2})\cap(\scrB_{R}^n)^{m_2}\bigr)
\\\label{R1R2limitLEM1pf1}
&=\prod_{j=1}^{m_2} \#\bigl((L^*\cap V^\perp)\cap(\scrB_R^n-\vecy_j)\bigr),
\end{align}
where $\vecy_1,\ldots,\vecy_{m_2}\in L^*$ are the column vectors of $y$.

Next note that $V=gW$ where $W$ is the linear span of the first $m_1$ standard unit vectors in $\R^n$;
hence $L^*\cap V^\perp=g^{-\mathsf{T}}(\Z^n\cap W^\perp)$,
and since the lattice $\Z^n\cap W^\perp$ has covolume one in $W^\perp$,
it follows that the lattice $L^*\cap V^\perp$ 
has covolume $\delta(g^{-\mathsf{T}},W^\perp)=\delta(g,W)$ in $V^\perp$
(the last equality holds by Lemma \ref{JACOBIANauxlem}).
Since $g=(x\hspace{5pt}*)$, this covolume equals $\dd(x)$.
Note also that for each $j$ and for $R$ large,
the ball 
$\scrB_R^n-\vecy_j$ intersects $V^\perp$ in a 
$(n-m_1)$-dimensional ball,
whose radius is $R\,(1-o(1))$ as $R\to\infty$.
Hence we conclude that $\#\bigl((L^*\cap V^\perp)\cap(\scrB_R^n-\vecy_j)\bigr)\sim\fV_{n-m_1}R^{n-m_1}\dd(x)^{-1}$
as $R\to\infty$,
and so
\begin{align}\label{R1R2limitLEM1pf2}
\lim_{R\to\infty} R^{-(n-m_1)m_2}\#\bigl((L^*)^{m_2}\cap S(\beta)'_x\cap(\scrB_{R}^n)^{m_2}\bigr)
=\frac{\fV_{n-m_1}^{m_2}}{\dd(x)^{m_2}}.
\end{align}

Hence the lemma will be proved if we can only prove that
the limits in \eqref{R1R2limitLEM1res} and \eqref{R1R2limitLEM1pf2} are equal.
Note that
$(L^*)^{m_2}\cap S(\beta)'_x\setminus S(\beta)_x=(y+(L^*\cap V^\perp)^{m_2})\setminus U_{m_2}$;
hence for any given $R$, the difference between the cardinalities
appearing in \eqref{R1R2limitLEM1res} and in \eqref{R1R2limitLEM1pf2}
is bounded above by $\sum_{\ell=1}^{m_2}E_\ell(R)$
where $E_\ell(R)$ is the number of matrices
$z\in (y+(L^*\cap V^\perp)^{m_2})\cap(\scrB_{R}^n)^{m_2}$
such that the $\ell$th column of $z$ lies in the ($\R$-)linear span of the other columns.
We will in fact prove that $E_\ell(R)=O(R^{(n-m_1)m_2-2})$ for each $\ell$;
this clearly implies the desired equality of the limits.

Given $\ell$, 
the columns of index $\neq\ell$ of a matrix
$z\in (y+(L^*\cap V^\perp)^{m_2})\cap(\scrB_R^n)^{m_2}$
can be fixed in $O(R^{(n-m_1)(m_2-1)})$ ways,
since the $i$th column must lie in
$(\vecy_i+(L^*\cap V^\perp))\cap\scrB_{R}^n$.
Hence it now suffices to prove that for every linear subspace 
$U$ of $\R^n$ 
which is spanned by vectors in $L^*$
and has dimension $\leq m_2-1$,
and for every $R\geq1$,
there exist at most $O(R^{n-m_1-2})$ vectors
in $U\cap(\vecy_\ell+(L^*\cap V^\perp))\cap\scrB_{R}^n$,
where the implied constant is independent of $U$ and $R$.
In fact we will prove the stronger fact that 
$\#(L^*\cap U \cap\scrB_R^n)\ll R^{n-m_1-2}$.
Note that $L^*\cap U $ is a full rank lattice in the Euclidean space $U$,
and its successive minima are bounded from below by the successive minima of $L^*$,
i.e.\ $\lambda_i(L^*\cap U )\geq\lambda_i(L^*)$ for each $i\leq\dim U$.
Hence by Lemma~\ref{LATTICECOUNTBOUNDlem1},
applied with $U$ in place of $\R^n$,
and using $R\geq1$,
we have 
$\#(L^*\cap U \cap\scrB_R^n)\ll_{L}R^{\dim U}$.
Also $R^{\dim U}\leq R^{m_2-1}\leq R^{n-m_1-2}$, since $m_1+m_2\leq n-1$.
The proof is complete.
\end{proof}

We now give the proof of Proposition \ref{cbetaexplPROP}.
Let $\beta\in M_{m_1,m_2}(\Z)$ be given.
As we have noted,
it follows from Lemmas \ref{GinvfunctionalLEM2},
\ref{GtransitiveonSbetaLEM} and \ref{etabetaGinvLEM} 
that there exists a constant $c_\beta\geq0$ such that
\begin{align}\label{cbetadef2}
F_\beta(\rho)=c_\beta \cdot\eta_\beta(\rho)
\end{align}
for all $\rho\in C_c(S(\beta))$.
Also, as we have noted,
the above equality can be viewed as an equality between two regular Borel measures on
$S(\beta)$;
hence the equality in \eqref{cbetadef2} holds more generally
for any bounded Borel measurable function $\rho$ on $S(\beta)$ of compact support.
In particular, 
given any bounded Borel subset $E\subset U_{m_1}$ and $R>0$,
the formula \eqref{cbetadef2} holds for the following function:
\begin{align}\label{rhoR1R2def}
\rho_{R}(x,y):=I\bigl(x\in E\bigr)
\cdot \dd(x)^{m_2}\cdot I\bigl(y\in(\scrB_{R}^n)^{m_2}\bigr).
\end{align}
We also have
\begin{align*}
\eta_\beta(\rho_{R})
&=\int_{U_{m_1}}\int_{S(\beta)_x}\rho_{R}(x,y)\,d\eta_{\beta,x}(y)\,\frac{dx}{\dd(x)^{m_2}}
\\
&=R^{m_2(n-m_1)}\int_{U_{m_1}}\int_{S(R^{-1}\beta)_x}\rho_{1}(x,y)\,d\eta_{R^{-1}\beta,x}(y)\,\frac{dx}{\dd(x)^{m_2}}
=R^{m_2(n-m_1)}\,\eta_{R^{-1}\beta}(\rho_{1}).
\end{align*}
(We substituted $y=Ry_{\new}$.)
Hence,
recalling also \eqref{MAINTHMpf6}, we have:
\begin{align}\label{R1R2limitLEM2varcons}
c_\beta\cdot\eta_{R^{-1}\beta}(\rho_{1})
=R^{-m_2(n-m_1)}
\int_{X_n}\sum_{x\in P_{L,m_1}}
\sum_{y\in (L^*)^{m_2}\cap S(\beta)_x}
\rho_{R}(x,y)\,d\mu(L)
\end{align}
For $E$ fixed, we will let $R\to\infty$ in the above equality.
The following lemma will be used to see that we may
then change the order of limit and integration.
For any $x\in M_{n,m_1}(\R)$, let us write $x^\perp$
for the orthogonal complement in $\R^n$ of 
$x\R^{m_1}$ (viz., the column span of $x$).
\begin{lem}\label{MAJORANTLEM1}
Define the function $\fM:X_n\to\R_{\geq0}$
by
\begin{align*}
\fM(L)=\sum_{x\in P_{L,m_1}}\sum_{y\in (L^*\cap x^\perp)^{m_2}}\rho_{1}(x,y).
\end{align*}
Then $\int_{X_n}\fM\,d\mu<\infty$,
and for all $L\in X_n$ and $R\geq1$ we have
\begin{align}\label{MAJORANTLEM1res1}
R^{-m_2(n-m_1)}
\sum_{x\in P_{L,m_1}}\sum_{y\in (L^*)^{m_2}\cap S(\beta)_x}
\rho_{R}(x,y)
\ll_{n} \fM(L).
\end{align}
\end{lem}
Note that the function $\fM$ depends on our choice of the fixed %
set $E$.
Note also that the implied constant in \eqref{MAJORANTLEM1res1} depends \textit{only} on $n$;
in particular it is independent of $L$ and $R$.
\begin{proof}
Since $E$ is bounded, there is a constant $B>0$ such that $\dd(x)^{m_2}\leq B$ for all $x\in E$,
and hence $\rho_{1}(x,y)\leq B$ on all $S(\beta)$.
Using this, together 
with the bound in \eqref{GinvfunctionalLEMpf1},
applied with $k_1=m_1$, $k_2=m_2$ and $K$ as 
$\overline{E}\times\bigl(\overline{\scrB_{1}^n}\bigr)^{m_2}$,
it follows that $\int_{X_n}\fM\:d\mu$ is finite.

In order to prove \eqref{MAJORANTLEM1res1} it suffices to prove
\begin{align}\label{MAJORANTLEM1pf1}
R^{-m_2(n-m_1)}\#\bigl((L^*)^{m_2}\cap S(\beta)_x\cap(\scrB_{R}^n)^{m_2}\bigr)
\ll_{n} \#\bigl(L^*\cap x^\perp\cap\scrB_1^n\bigr)^{m_2}
\end{align}
for all $L\in X_n$, $x\in P_{L,m_1}$ and $R\geq1$.
Indeed, \eqref{MAJORANTLEM1res1} follows by multiplying the inequality in
\eqref{MAJORANTLEM1pf1} by $\dd(x)^{m_2}$ and then adding over all 
$x\in P_{L,m_1}\cap E$.

In order to prove \eqref{MAJORANTLEM1pf1},
recall that
$S(\beta)_x\subset S(\beta)_x'$,
and as in the proof of Lemma \ref{R1R2limitLEM1}
there exist vectors $\vecy_1,\ldots,\vecy_{m_2}\in\R^n$
(depending on $x,\beta$)
such that \eqref{R1R2limitLEM1pf1} holds.
Hence it suffices to prove that
\begin{align}\label{MAJORANTLEM1pf2}
R^{-(n-m_1)}\#\bigl(L^*\cap x^\perp\cap(\scrB_R^n-\vecy)\bigr)\ll_n L^*\cap x^\perp\cap\scrB_1^n
\end{align}
for all $L\in X_n$, $x\in P_{L,m_1}$, $R\geq1$ and $\vecy\in\R^n$.

But $L^*\cap x^\perp$ is a (full rank) lattice in the Euclidean space $x^\perp$,
which has dimension $n-m_1$,
and for every $\vecy\in\R^n$, 
$\scrB_R^n-\vecy$ intersects $x^\perp$ in a ball of radius $\leq R$
(or the empty set).
Hence \eqref{MAJORANTLEM1pf2} follows from Lemma \ref{LATTICECOUNTBOUNDlem2},
applied with $x^\perp$ in the place of $\R^n$.
\end{proof}

Now we let
$R\to\infty$ in \eqref{R1R2limitLEM2varcons}.
In this limit we have
\begin{align*}
\eta_{R^{-1}\beta}(\rho_{1})
=\int_{E}\eta_{R^{-1}\beta,x}\bigl(S(R^{-1}\beta)_x\cap(\scrB_1^n)^{m_2}\bigr)\,dx
\to\fV_{n-m_1}^{m_2}\vol(E),
\end{align*}
by the Lebesgue dominated convergence theorem,
since for each fixed $x\in U_{m_1}$,
the measure
$\eta_{R^{-1}\beta,x}\bigl(S(R^{-1}\beta)_x\cap(\scrB_1^n)^{m_2}\bigr)$
tends to $\fV_{n-m_1}^{m_2}$ from below as $R\to\infty$.
On the other hand,
for each fixed $L\in X_n$
we have, by Lemma \ref{R1R2limitLEM1}
(using \eqref{rhoR1R2def} and the fact that 
$P_{L,m_1}\cap E$ is finite, since $E$ is bounded):
\begin{align*}
\lim_{R\to\infty}R^{-m_2(n-m_1)}\sum_{x\in P_{L,m_1}}
\sum_{\substack{y\in (L^*)^{m_2}\cap U_{m_2} %
\\ (x\transp y=\beta)}}
\rho_{R}(x,y)
=\fV_{n-m_1}^{m_2}\#\bigl(P_{L,m_1}\cap E\bigr).
\end{align*}
Hence, by changing order of the limit and the integration over $X_n$ in \eqref{R1R2limitLEM2varcons},
which is justified by the Lebesgue dominated convergence theorem
and Lemma \ref{MAJORANTLEM1},
we obtain:
\begin{align}\label{MAINTHMpf20}
c_\beta\cdot
\fV_{n-m_1}^{m_2}\vol(E)
=\fV_{n-m_1}^{m_2}
\int_{X_n}\#\bigl(P_{L,m_1}\cap E\bigr)\,d\mu(L)
=\frac{\fV_{n-m_1}^{m_2}\vol(E)}{\zeta(n)\cdots\zeta(n-m_1+1)},
\end{align}
where the last equality holds by
Theorem \ref{ROGERSbasicformulaTHM}.
Hence we obtain $c_\beta=\frac1{\zeta(n)\cdots\zeta(n-m_1+1)}$,
and thus, via \eqref{cbetadef2},
Proposition \ref{cbetaexplPROP} is proved.
Hence also Theorem \ref{MAINTHEOREMabs2} is proved.
\hfill$\square$ $\square$ $\square$

\subsection{Symmetries under transposition of $\beta$}
\label{betatranspsymmSEC}

We here prove the symmetry relations for $W(\beta)$ and $\eta_{\beta}$
mentioned in Remark \ref{DUALsymmrem};
see Lemma \ref{etabetasymmLEM} and
Lemma \ref{WsymmLEM}.

We first prove an auxiliary lemma:
\begin{lem}\label{etabetaasymptLEM}
For any fixed $m_1,m_2\in\Z^+$ and $\beta\in M_{m_1,m_2}(\Z)$,
\begin{align}\label{etabetaasymptLEMres}
\eta_\beta\Bigl(S(\beta)\cap\bigl((\scrB_R^n)^{m_1}\times(\scrB_R^n)^{m_2}\bigr)\Bigr)\sim C(n;m_1,m_2)\cdot R^{n(m_1+m_2)-2m_1m_2}\qquad
\text{as }\: R\to\infty,
\end{align}
where
\begin{align}\label{etabetaasymptLEMres2}
C(n;m_1,m_2)=\frac{\prod_{j=n-m_1+1}^n (j\fV_j)}{\prod_{j=n-m_1-m_2+1}^{n-m_2} (j\fV_j)}\, \fV_{n-m_1}^{\:m_2}\fV_{n-m_2}^{\:m_1}.
\end{align}
\end{lem}
(Here $\fV_d^{\: m}$ stands for $(\fV_d)^m$, i.e.\ $\fV_d$ raised to $m$.)
\begin{proof}
Using the definition \eqref{NUbetaDEFnew} and substituting $y=R\, y_{\new}$ in the inner integral,
it follows that the left-hand side of \eqref{etabetaasymptLEMres} equals
\begin{align*}
R^{m_2(n-m_1)}
\int_{ %
(\scrB_R^n)^{m_1}}\eta_{R^{-1}\beta,x}\Bigl(S(R^{-1}\beta)_x\cap(\scrB_1^n)^{m_1}\Bigr)\,\frac{dx}{\dd(x)^{m_2}}.
\end{align*}
Substituting here $x=R\, x_{\new}$,
and using the fact that $S(R^{-1}\beta)_{Rx}=S(R^{-2}\beta)_x$
with $\eta_{R^{-1}\beta,Rx}=\eta_{R^{-2}\beta,x}$,
the above expression becomes
\begin{align*}
R^{n(m_1+m_2)-2m_1m_2}
\int_{ %
(\scrB_1^n)^{m_1}}\eta_{R^{-2}\beta,x}\Bigl(S(R^{-2}\beta)_x\cap(\scrB_1^n)^{m_1}\Bigr)\,\frac{dx}{\dd(x)^{m_2}}.
\end{align*}
Here for each fixed $x\in U_{m_1}$, the integrand tends to
$\eta_{0,x}(S(0)_x\cap(\scrB_1^n)^{m_1})\cdot \dd(x)^{-m_2}
=\fV_{n-m_2}^{\: m_1}\dd(x)^{-m_2}$ 
from below as $R\to\infty$.
Hence by the Monotone Convergence Theorem,
we obtain \eqref{etabetaasymptLEMres} with 
$C(n;m_1,m_2)=\fV_{n-m_2}^{\: m_1}\int_{%
(\scrB_1^n)^{m_1}}\,\frac{dx}{\dd(x)^{m_2}}$.
Finally, by applying
\cite[Lemma 5.2]{aSaS2016}
two times,
we obtain \eqref{etabetaasymptLEMres2}.
\end{proof}
\begin{lem}\label{etabetasymmLEM}
Let $m_1,m_2\in\Z^+$ and $\beta\in M_{m_1,m_2}(\Z)$,
and let $J$ be the diffeomorphism
$\langle x,y\rangle\mapsto\langle y,x\rangle$ from $S(\beta)$ onto $S(\beta\transp)$.
Then $\eta_{\beta\transp}=J_*\eta_\beta$.
\end{lem}
\begin{proof}
Using the fact that the measure $\eta_{\beta}$ on $S(\beta)$ is $G$-invariant,
by Lemma \ref{etabetaGinvLEM},
together with the relation 
$J(gp)=g^{-\mathsf{T}} J(p)$
($\forall g\in G$, $p\in S(\beta)$),
it follows that $J_*\eta_{\beta}$ is a $G$-invariant measure on $S(\beta\transp)$.
Lemma \ref{etabetaGinvLEM} also implies that 
$\eta_{\beta\transp}$ is a $G$-invariant measure on $S(\beta\transp)$.
Hence as in the discussion below Lemma \ref{GtransitiveonSbetaLEM},
the two measures $J_*\eta_{\beta}$ and $\eta_{\beta\transp}$ 
are constant multiples of each other, %
and to complete the proof of the lemma, it now suffices to prove that
$J_*\eta_{\beta}(E_R)\sim\eta_{\beta\transp}(E_R)$ as $R\to\infty$,
when $E_R:=S(\beta\transp)\cap\bigl((\scrB_R^n)^{m_2}\times(\scrB_R^n)^{m_1}\bigr)$.
However, by Lemma \ref{etabetaasymptLEM} we have
\begin{align*}
J_*\eta_{\beta}(E_R)=\eta_{\beta}(J^{-1}E_R)\sim C(n;m_1,m_2)\cdot R^{n(m_1+m_2)-2m_1m_2}
\end{align*}
and also
\begin{align*}
\eta_{\beta\transp}(E_R)\sim C(n;m_2,m_1)\cdot R^{n(m_1+m_2)-2m_1m_2}
\end{align*}
as $R\to\infty$.
Hence the lemma follows from the fact that
$C(n;m_1,m_2)=C(n;m_2,m_1)$.
\end{proof}

Next we discuss the weight function $W(\beta)$.
Let us first note a basic invariance relation.
\begin{lem}\label{Winvlem}
Let $m_1,m_2\in\Z^+$ and $\beta\in M_{m_1,m_2}(\Z)$.
Then
\begin{align}\label{Winvlemres}
W(\gamma_1\beta\gamma_2)=W(\beta)\qquad\forall\gamma_1\in\GL_{m_1}(\Z),\:\gamma_2\in\GL_{m_2}(\Z).
\end{align}
\end{lem}
\begin{proof}
Let $U$ be the set of non-singular matrices in $M_{m_1,m_1}(\Z)$,
let $\Delta=\GL_{m_1}(\Z)\subset U$,
and let $\Delta\bs U=\{\Delta A\col A\in U\}$
be the set of all right $\Delta$-cosets contained in $U$.
Note that the set $\fW_{m_1}$ contains exactly one element in
each coset $x\in\Delta\bs U$.
Furthermore, for each $x\in\Delta\bs U$, 
the relation
$A^{-\mathsf{T}}\beta\in M_{m_1,m_2}(\Z)$ either holds for
\textit{all} or \textit{no} element $A\in x$,
and we denote by $(\Delta\bs U)_\beta$ the subset
$\{x\in\Delta\bs U\col A^{-\mathsf{T}}\beta\in M_{m_1,m_2}(\Z)\:\forall A\in x\}$.
We also write $|\det x|$ for the common value of
the absolute determinant $|\det A|$ for all $A\in x$.
With this notation,
the formula in \eqref{Wdef} may be expressed as %
\begin{align*}
W(\beta)=\frac{\sum_{x\in(\Delta\bs U)_\beta}|\det x|^{m_2-n}}{\prod_{j=n-m_1+1}^n\zeta(j)}.
\end{align*}
Now $\Delta$ acts on $\Delta\bs U$ by right multiplication;
$(\Delta A)\gamma:=\Delta(A\gamma)$ for any $A\in U$ %
and $\gamma\in\Delta$.
One verifies that 
$(\Delta\bs U)_{\gamma_1\beta\gamma_2}=
\{x\gamma_1^{\mathsf{T}} \col x\in(\Delta\bs U)_\beta\}$
for any $\gamma_1\in\Delta$ and $\gamma_2\in\GL_{m_2}(\Z)$.
Thus
\begin{align*}
\sum_{x\in(\Delta\bs U)_{\gamma_1\beta\gamma_2}}|\det x|^{m_2-n}
=\sum_{x\in(\Delta\bs U)_{\beta}}|\det (x\gamma_1^{\mathsf{T}})|^{m_2-n}
=\sum_{x\in(\Delta\bs U)_{\beta}}|\det x|^{m_2-n}.
\end{align*}
Hence $W(\gamma_1\beta\gamma_2)=W(\beta)$.
\end{proof}

\begin{lem}\label{WsymmLEM}
For any $m_1,m_2\in\Z^+$ and $\beta\in M_{m_1,m_2}(\Z)$,
$W(\beta\transp)=W(\beta)$.
\end{lem}
\begin{proof}
Replacing $\beta$ by $\beta\transp$ if necessary,
we may assume that $m_1\leq m_2$.
By the Smith Normal Form Theorem,
there exist
$\gamma_1\in\GL_{m_1}(\Z)$ and $\gamma_2\in\GL_{m_2}(\Z)$
such that, in block matrix notation,
$\gamma_1\beta\gamma_2=\bigl(D \hspace{5pt} 0\bigr)$ for some diagonal matrix $D\in M_{m_1,m_1}(\Z)$.
If $m_1=m_2$ then
$\gamma_1\beta\gamma_2=D=D\transp=\gamma_2^{\mathsf{T}}\beta\transp\gamma_1^{\mathsf{T}}$,
and so
$W(\beta\transp)=W(\beta)$
by Lemma \ref{Winvlem}.

Hence from now on we assume $m_1<m_2$.
Transposing the relation
$\gamma_1\beta\gamma_2=\bigl(D \hspace{5pt} 0\bigr)$
and multiplying from the left by an appropriate permutation matrix,
it follows that there exists 
$\gamma_3\in\GL_{m_2}(\Z)$ such that
$\gamma_3\beta\transp\gamma_1^{\mathsf{T}}=\cmatr{0}{D}=:D'$
in $M_{m_2,m_1}(\Z)$,
so that $W(\beta\transp)=W(D')$ by %
Lemma \ref{Winvlem}.
Now for any 
$A=\matr{A_1}{A_{12}}{0}{A_2}\in\fW_{m_2}$
(with blocks $A_1,A_{12},A_2$ of sizes $(m_2-m_1)\times(m_2-m_1)$,
$(m_2-m_1)\times m_1$ and $m_1\times m_1$, respectively),
one verifies that $A\in\fW_{D'}$ if and only if $A_2\in\fW_D$;
furthermore, for any given $A_2\in\fW_D$
there are exactly
$(\det A_2)^{m_2-m_1}$ possibilities for %
$A_{12}$ which, together with 
$A_1$ freely chosen in $\fW_{m_2-m_1}$,
makes $A\in\fW_{D'}$.
Hence
\begin{align*}
\sum_{A\in\fW_{D'}}(\det A)^{m_1-n}
=\sum_{A_2\in\fW_D}(\det A_2)^{m_2-m_1}\sum_{A_1\in\fW_{m_2-m_1}}\bigl(\det A_1\,\det A_2\bigr)^{m_1-n}
\hspace{50pt}
\\
=\sum_{A_2\in\fW_D}(\det A_2)^{m_2-n} \prod_{j=n-m_2+1}^{n-m_1}\zeta(j).
\end{align*}
Using this, and the fact that $\fW_{(D\:0)}=\fW_D$,
we get
\begin{align*}
W(\beta\transp)
=W(D')
=\frac{\sum_{A\in\fW_{D'}}(\det A)^{m_1-n}}{\prod_{j=n-m_2+1}^n\zeta(j)}
=\frac{\sum_{A\in\fW_D}(\det A)^{m_2-n}}{\prod_{j=n-m_1+1}^n\zeta(j)}
=W\bigl((D\hspace{5pt}0)\bigr)
=W(\beta).
\end{align*}
\end{proof}

\section{Outline of an alternative proof %
using Poisson summation}
\label{proof1sec}

In this section we give an outline of an alternative proof of our main result,
Theorem \ref{MAINTHEOREMabs2}.

The main idea of the proof is to apply the Poisson summation formula
in both %
sides of \eqref{MAINTHEOREMabs2res}.
The manipulations which we will describe can be justified for any function
$\rho:(\R^n)^{k_1}\times(\R^n)^{k_2}\to\C$
such that $\rho$ and all its derivatives up to a sufficiently high order are of 
rapid polynomial decay
(cf.\ Remark \ref{ROGERSFORMULATHMabsconvrem}).
Once the equality in \eqref{MAINTHEOREMabs2res}
has been proved for this family of test functions,
it follows that we have a corresponding equality of regular Borel measures on 
$(\R^n)^{k_1}\times(\R^n)^{k_2}$
(cf.\ Lemma \ref{GinvfunctionalLEM}),
and in particular \eqref{MAINTHEOREMabs2res} also holds for the class of
functions appearing in the statement of Theorem \ref{MAINTHEOREMabs2}.

In the left-hand side of \eqref{MAINTHEOREMabs2res}, 
we apply Poisson summation
to the sums over $L^*$,
to obtain
\begin{align}\label{MTHMNEWpf100}
\int_{X_n}\sum_{\vecv_1,\ldots,\vecv_{k}\in L}\trho(\vecv_1,\ldots,\vecv_{k})\,d\mu(L),
\end{align}
with
\begin{align}\label{trhoDEF}
\trho(x,v):=\int_{M_{n,k_2}(\R)}\rho(x,y)e(-\Tr(v\transp y))\,dy.
\end{align}
The integral in \eqref{MTHMNEWpf100} may be evaluated using Rogers' mean value formula,
Theorem \ref{ROGERSFORMULATHMp}.
The main work of the proof now consists in applying
Poisson summation to the sum over $\beta$ 
in the right-hand side of \eqref{MAINTHEOREMabs2res},
and seeking to match the resulting expressions.

Let us fix an arbitrary choice of $m_1,B_1,m_2,B_2$ appearing in 
\eqref{MAINTHEOREMabs2res}.
In the following we will assume $m_1>0$ and $0<m_2<k_2$,
but with appropriate interpretations, the argument which we present
can be seen to also cover the remaining cases
(cf.\ also Remark \ref{emptymatricesREM}).
Using the definitions of $\eta_\beta$ and $W(\beta)$
(cf.\ \eqref{Wdef} and \eqref{NUbetaDEFnew}),
and substituting
$\beta=A\transp\gamma$ 
(with $A\in\fW$ and $\gamma\in M_{m_1,m_2}(\Z)$),
we find that
the sum over $\beta$ in \eqref{MAINTHEOREMabs2res} can be rewritten as
\begin{align}\label{MTHMNEWpf1}
\sum_{A\in\fW}\sum_{\gamma\in M_{m_1,m_2}(\Z)}\frac{(\det A)^{m_2-n}}{\zeta(n)\zeta(n-1)\cdots\zeta(n-m_1+1)}
\int_{U_{m_1}} F_x(A\transp\gamma)\,\frac{dx}{\dd(x)^{m_2}},
\end{align}
where, for any $x\in U_{m_1}$,
the function $F_x$ is given by
\begin{align}\label{MTHMNEWpf1help}
F_x:M_{m_1,m_2}(\R)\to\C;
\qquad
F_x(\beta)=\int_{S(\beta)_x}\rho\bigl(xB_1^{\mathsf{T}},yB_2^{\mathsf{T}}\bigr)\,d\eta_{\beta,x}(y).
\end{align}
In \eqref{MTHMNEWpf1} we move the sum over $\gamma$ inside the integral,
and apply the Poisson summation formula to
the resulting inner sum,
$\sum_{\gamma %
}F_x(A\transp\gamma)$.
For this we need to compute the Fourier transform of $F_x$,
\begin{align*}
\hF_x(\xi):=\int_{M_{m_1,m_2}(\R)}F_x(\beta)e(-\Tr(\xi\transp  \beta))\,d\beta
\qquad
(\xi\in M_{m_1,m_2}(\R)).
\end{align*}
Noticing that $S(\beta)_x'=x(x\transp x)^{-1}\beta+(x^\perp)^{m_2}$
we may rewrite the definition of $F_x$ in \eqref{MTHMNEWpf1help}
as an integral over $y'\in x^\perp$,
with the integrand being $\rho(xB_1^{\mathsf{T}},(x(x\transp x)^{-1}\beta+y')B_2^{\mathsf{T}})$.
One verifies that the map taking
$\langle \beta,y'\rangle$ to $(x(x\transp x)^{-1}\beta+y')B_2\transp$
is a linear bijection from 
$M_{m_1,m_2}(\R)\oplus (x^\perp)^{m_2}$
onto the subspace
$(V_{B_2}^n)\transp$ of $M_{n,k_2}(\R)$,
and we thus obtain
\begin{align*}
\hF_x(\xi):=\dd(x)^{m_2}\dd(B_2)^{-n}
\int_{(V_{B_2}^n)\transp}
\rho\bigl(xB_1^{\mathsf{T}},w\bigr)\,e\bigl(-\Tr\bigl(\xi\transp  x\transp wB_2(B_2^{\mathsf{T}}B_2)^{-1}\bigr)\bigr)\,dw.
\end{align*}
Here note that 
$\Tr\bigl(\xi\transp  x\transp wB_2(B_2^{\mathsf{T}}B_2)^{-1}\bigr)
=\Tr\bigl(\bigl(x\xi(B_2^{\mathsf{T}}B_2)^{-1}B_2^{\mathsf{T}}\bigr)\transp w\bigr)$;
hence the last integral is the Fourier transform of the
function $\rho\bigl(xB_1^{\mathsf{T}},\cdot\,\bigr)$ restricted to $(V_{B_2}^n)\transp$,
evaluated at the point $x\xi(B_2^{\mathsf{T}}B_2)^{-1}B_2^{\mathsf{T}}$.
In view of \cite[Thm.\ 4.42 and the ensuing remark]{gF95a},
this can be re-expressed as 
an integral over a translate of the orthogonal complement of 
$(V_{B_2}^n)\transp$ in $M_{n,k_2}(\R)$,
viz., $((V_{B_2}^\perp)^n)\transp$, of the 'full' Fourier transform $\trho$ (cf.\ \eqref{trhoDEF}):
\begin{align}\label{hFxintermsoftrho}
\hF_x(\xi):=\dd(x)^{m_2}\dd(B_2)^{-n}
\int_{((V_{B_2}^\perp)^n)\transp}
\trho\bigl(xB_1^{\mathsf{T}},x\xi(B_2^{\mathsf{T}}B_2)^{-1}B_2^{\mathsf{T}}+u\bigr)\,du.
\end{align}
Hence, by the Poisson summation formula,
we conclude that \eqref{MTHMNEWpf1} equals:
\begin{align}\label{MTHMNEWpf2}
\sum_{A\in\fW}\frac{(\det A)^{-n}}{\zeta(n)\zeta(n-1)\cdots\zeta(n-m_1+1)}
\,\dd(B_2)^{-n}
\hspace{180pt}
\\\notag
\times\,\int_{M_{n,m_1}(\R)}
\sum_{\xi\in M_{m_1,m_2}(\Z)}
\int_{((V_{B_2}^\perp)^n)\transp}
\trho\Bigl(xB_1^{\mathsf{T}}, xA^{-1} \,\xi(B_2^{\mathsf{T}}B_2)^{-1}B_2^{\mathsf{T}}+u\Bigr)\,du\,dx.
\end{align}

Next, in \eqref{MTHMNEWpf2}, we substitute $\xi:=A \xi'$,
and then move the summation over $\xi'$ 
to the leftmost position.
Then $\xi'$ runs through $M_{m_1,m_2}(\Q)$,
and for each given   %
$\xi'\in M_{m_1,m_2}(\Q)$,
$A$ runs through all $A\in\fW$ satisfying $A\xi'\in M_{m_1,m_2}(\Z)$.
For any $\xi\in M_{m_1,m_2}(\Q)$ we write $q_\xi\in\Z^+$ for the least common denominator of the entries of $\xi$,
that is, the smallest positive integer $q$ such that $q\xi\in M_{m_1,m_2}(\Z)$.
Then the last relation,
$A\xi'\in M_{m_1,m_2}(\Z)$,
is equivalent with
$A\, q_{\xi'}\,\xi'\equiv 0\mod q_{\xi'}$.
We will need the following lemma, applied with $\theta:=q_{\xi'}\,\xi'\mod q_{\xi'}$.

\begin{lem}\label{LINALGlem}
Let $q$ be a positive integer and let $\theta\in M_{m_1,m_2}(\Z/q\Z)$.
Then
\begin{align}\label{LINALGCONJ3}
\sum_{\substack{A\in\fW\\ A\theta\equiv0\mod q}}
\frac{(\det A)^{-n}}{\zeta(n)\zeta(n-1)\cdots\zeta(n-m_1+1)}
=\Bigl(\frac{N}{q^{m_1}}\Bigr)^{n},
\end{align}
where
\begin{align}\label{LINALGCONJ4}
N=\#\bigl\{\veca\in(\Z/q\Z)^{m_1}\col\theta\transp\veca=\bn\text{ in $(\Z/q\Z)^{m_2}$}\bigr\}.
\end{align}
\end{lem}
\begin{proof}
Let $R$ be the ring $\Z/q\Z$.
The lemma follows from a study of the solution set
\begin{align*}
S:=\bigl\{\veca\in R^{m_1}\col\veca\theta=\bn\text{ in }R^{m_2}\bigr\}.
\end{align*}
(Note that the equation $\veca\theta=\bn$ is equivalent with the equation $\theta\transp\veca=\bn$
appearing in \eqref{LINALGCONJ4}, up to switching between row and column vectors;
$\veca\leftrightarrow\veca\transp$.)
For $j\in\{1,\ldots,m_1\}$, set
\begin{align*}
I_j:=\bigl\{a_j\in R\col\exists a_{j+1},\ldots,a_{m_1}\in R\text{ such that }
(0\,\cdots\,0\,a_j\,a_{j+1}\,\cdots\,a_{m_1})\in S\bigr\}.
\end{align*}
Then $I_j$ is a subgroup of $\langle R,+\rangle$;
hence there is some $d_j\mid q$ such that
$I_j=d_jR$.
For any $d\mid q$ let us write $A_d:=\{0,1,\ldots,d-1\}\subset R$.
It is now straightforward to verify that
\begin{align}\label{LINALGCONJpf10}
\left\{
\text{\parbox{390pt}{for any $j\in\{1,\ldots,m_1\}$,
any $b_j\in d_jR$ and any $b_{j+1},\ldots,b_{m_1}\in R$,
there is a \textit{unique} solution
$\veca\in S$ in the 'box'
$\{0\}^{j-1}\times\{b_j\}\times (b_{j+1}+A_{d_{j+1}})\times\cdots\times(b_{m_1}+A_{d_{m_1}})$.}}\right.
\end{align}

In particular, noticing that $R^{m_1-1}$ can be partitioned into 
$\prod_{j=2}^{m_1}(q/d_j)$ boxes of the form $(b_2+A_{d_2})\times\cdots\times(b_{m_1}+A_{d_{m_1}})$,
it follows from \eqref{LINALGCONJpf10} applied with $j=1$ that
\begin{align}\label{LINALGCONJpf1}
N=\#S=\#(d_jR)\cdot\prod_{j=2}^{m_1}(q/d_j)=\frac{q^{m_1}}{\prod_{j=1}^{m_1}d_j}.
\end{align}

Next we consider the sum in \eqref{LINALGCONJ3}. %
Note that the equation  $A\theta\equiv0\mod q$ means that for each $j$,
row number $j$ of $A$ satisfies the equation appearing in the definition of $I_j$.
It follows that for any matrix $A$ appearing in the sum in 
\eqref{LINALGCONJ3},
the diagonal entries of $A$ must be
$d_1c_1,d_2c_2,\ldots,d_{m_1}c_{m_1}$ for some positive integers $c_1,\ldots,c_{m_1}$,
and having fixed this diagonal, 
it follows via \eqref{LINALGCONJpf10}
that for each $j$, the $j$th row of $A$ can be choosen
in $\prod_{i=j+1}^{m_1}c_j$ ways,
and so the total number of matrices $A$ with this diagonal is
$\prod_{j=1}^{m_1}\prod_{i=j+1}^{m_1}c_j
=\prod_{j=2}^{m_1}c_j^{j-1}$.
Hence the left-hand side of \eqref{LINALGCONJ3} equals
\begin{align*}
\sum_{c_1=1}^{\infty}\cdots \sum_{c_{m_1}=1}^{\infty}\biggl(\prod_{j=2}^{m_1}c_j^{j-1}\biggr)
\times\frac{\prod_{j=1}^{m_1}(c_jd_j)^{-n}}{\zeta(n)\zeta(n-1)\cdots\zeta(n-m_1+1)}
=\Bigl(\prod_{j=1}^{m_1}d_j\Bigr)^{-n}.
\end{align*}
Combining this with \eqref{LINALGCONJpf1} we obtain \eqref{LINALGCONJ3}.
\end{proof}
By Lemma \ref{LINALGlem}
and the discussion preceding it,
\eqref{MTHMNEWpf2} equals:   %
\begin{align}\label{MTHMNEWpf3}
\sum_{\xi\in M_{m_1,m_2}(\Q)}\biggl(\frac{N_{\xi}}{q_{\xi}^{m_1}}\biggr)^n
\dd(B_2)^{-n}\,
\int_{M_{n,m_1}(\R)}
\int_{((V_{B_2}^\perp)^n)\transp}
\trho\Bigl(xB_1^{\mathsf{T}}, x\xi(B_2^{\mathsf{T}}B_2)^{-1}B_2^{\mathsf{T}}+u\Bigr)\,du\,dx,
\end{align}
where 
\begin{align}\label{Nxidef}
N_\xi:=\#\{\veca\in(\Z/q_{\xi}\Z)^{m_1}\col q_{\xi}\xi\transp\veca\equiv\bn\mod q_{\xi}\}.
\end{align}

Set $m_2':=k_2-m_2$.
Recall that we are assuming $0<m_2<k_2$; thus $0<m_2'<k_2$.
Using the fact that the map $V\mapsto V^\perp$
gives a bijection from the family of $m_2$-dimensional rational subspaces of $\R^{k_2}$
onto the family of $m_2'$-dimensional rational subspaces of $\R^{k_2}$,
it follows that
for every $B_2\in A_{k_2,m_2}$ there exists a unique
matrix $\tB_2\in A_{k_2,m_2'}$ satisfying $V_{B_2}^\perp=V_{\tB_2}$,
and the map $B_2\mapsto\tB_2$ is a bijection from 
$A_{k_2,m_2}$ onto $A_{k_2,m_2'}$.
With this notation,
the map $y\mapsto y{\tB_2}\transp$ is a linear bijection from $M_{n,m_2'}(\R)$
onto $((V_{B_2}^\perp)^n)\transp$,
which scales volume by a factor $\dd(\tB_2)^n$,
and by \cite[Sec.\ 2 (Corollary)]{wS68}
we have
$\dd(\tB_2)=\dd(B_2)$.
Hence \eqref{MTHMNEWpf3} equals
\begin{align}\label{MTHMNEWpf3rep}
\sum_{\xi\in M_{m_1,m_2}(\Q)}\biggl(\frac{N_{\xi}}{q_{\xi}^{m_1}}\biggr)^n
\int_{M_{n,m_1}(\R)}
\int_{M_{n,m_2'}(\R)}
\trho\Bigl(xB_1^{\mathsf{T}}, x\xi(B_2^{\mathsf{T}}B_2)^{-1}B_2^{\mathsf{T}}+y{\tB_2}\transp\Bigr)\,dy\,dx,
\end{align}
Set $k:=k_1+k_2$,
and view $\trho$ as a function on $M_{n,k}(\R)$
by identifying any pair $\langle u,v\rangle$ 
in $M_{n,k_1}(\R)\times M_{n,k_2}(\R)$ with the 
block matrix $\bigl(u\: v\bigr)$ in $M_{n,k}(\R)$.
We also set
\begin{align}\label{alphaintermsofxi}
\alpha:=(B_2^{\mathsf{T}} B_2)^{-1}\xi\transp\in M_{m_2,m_1}(\Q).
\end{align}
Using this notation, \eqref{MTHMNEWpf3rep}
equals
\begin{align}\label{MTHMNEWpf4}
\sum_{\xi\in M_{m_1,m_2}(\Q)}\biggl(\frac{N_{\xi}}{q_{\xi}^{m_1}}\biggr)^n
\int_{M_{n,m_1+m_2'}(\R)}\trho\left(z
{\matr{B_1}0{B_2\alpha}{\tB_2}}^{\!\mathsf{T}\:}
\right)\,dz.
\end{align}

To sum up, we have prove that
the sum over $\beta$ in \eqref{MAINTHEOREMabs2res} 
equals \eqref{MTHMNEWpf4}.
Next, as in \eqref{MAINTHEOREMabs2res},
this expression should be added over all
$B_1\in A_{k_1,m_1}$ and $B_2\in A_{k_2,m_2}$.
The following lemma will allow us
to rearrange in a convenient way the triple sum over $B_1,B_2,\xi$ 
which then appears.
Let us set $P_1:=\bigl(I_{k_1}\: 0\bigr)\in M_{k_1,k}(\Z)$,
so that $P_1(v_1,\ldots,v_k)\transp=(v_1,\ldots,v_{k_1})\transp$ for all $\vecv=(v_1,\ldots,v_k)\transp\in\R^k$.

\begin{lem}\label{B1B2alphabijLEM}
Given any $B_1\in A_{k_1,m_1}$, $B_2\in A_{k_2,m_2}$
and $\alpha\in M_{m_2,m_1}(\Q)$,
there exist unique matrices $B_3\in A_{k,m_1+m_2'}$ and 
$J\in\GL_{m_1+m_2'}(\R)$
satisfying
\begin{align}\label{B1B2alphabijLEMres}
\matr{B_1}0{B_2\alpha}{\tB_2}=B_3J,
\end{align}
and the map $\langle B_1,B_2,\alpha\rangle\mapsto B_3$ 
defined by this relation
is a bijection from
$A_{k_1,m_1}\times A_{k_2,m_2}\times M_{m_2,m_1}(\Q)$
onto the set of all $B_3\in A_{k,m_1+m_2'}$ satisfying 
$\dim P_1V_{B_3} %
=m_1$.
Furthermore, when the above relation holds, we also have
\begin{align}\label{B1B2alphabijLEMres2}
|\det J|=N_\xi/q_{\xi}^{m_1}
\qquad\text{with }\:
\xi=%
\alpha\transp B_2\transp B_2.
\end{align}
\end{lem}
(Note that the 
formula for $\xi$ in \eqref{B1B2alphabijLEMres2} is equivalent with
\eqref{alphaintermsofxi}.)
\begin{proof}[Proof (outline).]
If $D$ is the matrix in the left-hand side of 
\eqref{B1B2alphabijLEMres},
then $D\R^{m_1+m_2'}$ is a 
rational linear subspace of $\R^k$ of dimension $m_1+m_2'$,
and thus there exists a unique matrix 
$B_3\in A_{k,m_1+m_2'}$
satisfying $D\R^{m_1+m_2'}=V_{B_3}$,
and then also the unique existence of $J$ follows.
Next, when \eqref{B1B2alphabijLEMres} holds, we also have
$P_1B_3=P_1DJ^{-1}=\bigl(B_1\:\:0\bigr)J^{-1}$
and hence
$P_1V_{B_3}=V_{B_1}$;
in particular $\dim P_1V_{B_3}=m_1$.

To complete the proof of the bijectivity statement,
it now suffices to prove that if $V$ is an arbitrary 
rational $(m_1+m_2')$-dimensional linear subspaces $V\subset\R^k$ satisfying $\dim P_1V=m_1$,
then there exists unique 
triple $\langle B_1,B_2,\alpha\rangle$
in $A_{k_1,m_1}\times A_{k_2,m_2}\times M_{m_2,m_1}(\Q)$
such that $D\R^{m_1+m_2'}=V$,
where $D=D(B_1,B_2,\alpha)$ is the matrix in the left-hand side of \eqref{B1B2alphabijLEMres}.
Given $V$, we set
$K:=\ker P_1\big|_{V}$;
then there exists a uniquely determined matrix $C\in M_{k_2,k_1}(\Q)$
such that $W:=K^\perp\cap V=\left\{\cmatr{\vecu}{C\vecu}\col\vecu\in P_1V\right\}$.
On the other hand, for any 
$D=D(B_1,B_2,\alpha)$ as above we have
\begin{align}\label{B1B2alphabijLEMpf2}
D\R^{m_1+m_2'}=\left\{\cmatr{B_1\vecx}{B_2\alpha\vecx}\col\vecx\in\R^{m_1}\right\}
\oplus\left\{\cmatr{\bn}{\tB_2\vecy}\col\vecy\in\R^{m_2'}\right\},
\end{align}
and this is in fact an orthogonal direct sum,
and the second term is equal to $\ker P_1\big|_{D\R^{m_1+m_2'}}$.
Hence we conclude that $D\R^{m_1+m_2'}=V$
holds if and only if the two terms in 
\eqref{B1B2alphabijLEMpf2} are equal to $W$ and $K$, respectively,
and this, in turn, holds if and only if
$V_{B_1}=P_1V$, $V_{\tB_2}$ equals the image of $K\subset\{\bn\}\oplus\R^{k_2}$ in $\R^{k_2}$,
and $B_2\alpha=CB_1$.
The first two of these relations are satisfied for unique choices of
$B_1\in A_{k_1,m_1}$ and $B_2\in A_{k_2,m_2}$, respectively,
and then the third relation holds for a unique choice of 
$\alpha\in M_{m_2,m_1}(\Q)$.

It remains to prove \eqref{B1B2alphabijLEMres2}.
When \eqref{B1B2alphabijLEMres} holds,
one verifies by a direct computation that 
\begin{align}\label{B1B2alphabijdetpf1}
J^{-1}\Z^{m_1+m_2'}
=\left\{\cmatr{\vecx}{\vecy}\col \vecx\in\Z^{m_1},\vecy\in\R^{m_2'},\:
B_2\alpha \vecx+\tB_2\vecy\in\Z^{k_2}\right\}.
\end{align}
Let $L$ be the image of $J^{-1}\Z^{m_1+m_2'}$ under the projection
onto the first $m_1$ coordinates.
Using 
$\tB_2\R^{m_2'}=V_{B_2}^\perp=\{\vecw\in \R^{k_2}\col B_2^{\mathsf{T}}\vecw=\bn\}$,
it follows that $L$ equals the set of those $\vecx\in\Z^{m_1}$
for which there exists some $\vecu\in\Z^{k_2}$
satisfying $B_2^{\mathsf{T}} (\vecu-B_2\alpha\vecx)=0$.
In other words,
$L$ is the set of $\vecx\in\Z^{m_1}$
satisfying $\xi\transp\vecx\in B_2^{\mathsf{T}}\Z^{k_2}$, 
with $\xi$ as in \eqref{B1B2alphabijLEMres2}.
Next, using the fact that $B_2$ can be extended to a matrix $B_2'\in\GL_{k_2}(\Z)$
with block decomposition $B_2'=(B_2\hspace{5pt}*)$,
it follows that $B_2^{\mathsf{T}}\Z^{k_2}=\Z^{m_2}$.
Hence $\xi\transp\vecx\in B_2^{\mathsf{T}}\Z^{k_2}$ is equivalent with 
$\xi\transp\vecx\in\Z^{k_2}$,
which in turn is equivalent with 
$q_{\xi}\xi\transp\vecx\equiv\bn\mod q_{\xi}$.
It follows that $L$ is a full-dimensional sub-lattice of $\Z^{m_1}$
of index $q_{\xi}^{m_1}/N_\xi$
(cf.\ \eqref{Nxidef}).
Furthermore, given any $\vecx\in\Z^{m_1}$ and $\vecy_0\in\R^{m_2'}$ such that
$B_2\alpha\vecx+\tB_2\vecy_0\in\Z^{k_2}$,
we have
$\{\vecy\in\R^{m_2'}\col B_2\alpha\vecx+\tB_2\vecy\in\Z^{k_2}\}
=\vecy_0+\Z^{m_2'}$.
Using these facts, it follows that 
\begin{align}\label{detJinvcomp}
\bigl|\det J^{-1}\bigr|=\vol(\R^{m_1+m_2'}/J^{-1}\Z^{m_1+m_2'})
=\vol(\R^{m_1}/L)=q_{\xi}^{m_1}/N_\xi,
\end{align}
and the formula \eqref{B1B2alphabijLEMres2} is proved.
\end{proof}

Recall that we are considering a fixed choice of
$m_1>0$ and $0<m_2<k_2$.
Using Lemma~\ref{B1B2alphabijLEM},
and noticing that when \eqref{B1B2alphabijLEMres} holds then the integral in
\eqref{MTHMNEWpf4} equals
$\int\trho(zJ\transp B_3^{\mathsf{T}})\,dz
=\bigl(N_\xi/q_{\xi}^{m_1}\bigr)^{-n}\int\trho(wB_3^{\mathsf{T}})\,dw$,
it follows that the sum 
over all $B_1\in A_{k_1,m_1}$ and $B_2\in A_{k_2,m_2}$
of the expression in 
\eqref{MTHMNEWpf4} 
equals
\begin{align}\label{MTHMNEWpf1000} 
\sum_{\substack{B\in A_{k,m_1+m_2'}\\(\dim P_1V_B=m_1)}}
\int_{M_{n,m_1+m_2'}(\R)}\trho\bigl(wB^{\mathsf{T}}\bigr)\,dw
\end{align}
By appropriate modifications of the previous argument,
the identity just stated can be shown to hold also in the remaining cases
when $m_1=0$ and/or $m_2\in\{0,k_2\}$.
(In the special case $m_1=0$, $m_2=k_2$,
the expression in \eqref{MTHMNEWpf1000}  should be interpreted to mean $\trho(0)$.)
Finally, adding the expression in \eqref{MTHMNEWpf1000} over
all combinations of
$m_1\in\{0,\ldots,k_1\}$
and $m_2\in\{0,\ldots,k_2\}$
we arrive at the conclusion that the right-hand side of 
\eqref{MAINTHEOREMabs2res} equals
\begin{align*}
\sum_{m=1}^k\sum_{B\in A_{k,m}}\int_{M_{n,m}(\R)}\trho(xB\transp)\,dx
+\trho(0).
\end{align*}
By Rogers' formula, Theorem \ref{ROGERSFORMULATHMp},
this equals \eqref{MTHMNEWpf100},
i.e.\ the left-hand side of \eqref{MAINTHEOREMabs2res},
and Theorem \ref{MAINTHEOREMabs2} is proved.
\hfill$\square$ $\square$ $\square$

\section{Application to the joint limit distribution of 
$\{\mathcal V_j(L)\}_{j=1}^{\infty}$
and $\{\mathcal V_j(L^*)\}_{j=1}^{\infty}$}
\label{applicationSEC}

\subsection{Joint moments of counting functions $N_j(L)$ and $\tN_j(L^*)$}\label{JMsec}
Given any numbers
$0<V_1\leq V_2\leq\cdots\leq V_{k_1}$
and $0<W_1\leq W_2\leq\cdots\leq W_{k_2}$,
for any $L\in X_n$ we denote by $N_j(L)$ the number of non-zero lattice points of $L$
in the open $n$-ball of volume $V_j$ centered at the origin,
and we denote by $\tN_j(L^*)$ the number of non-zero lattice points of $L^*$
in the $n$-ball of volume $W_j$ centered at the origin.
The main step in the proof of Theorem \ref{JOINTPOISSONTHEOREM}
is the following theorem concerning the joint moments of
the counting functions $N_j(L)$ and $\tN_j(L^*)$,
for $L$ random in $(X_n,\mu)$.

For any $k\in\Z_{\geq0}$ we denote by
$\scrP(k)$ the set of partitions of the set $\{1,\ldots,k\}$.
Thus, for example, $\scrP(1)=\{\{\{1\}\}\}$
and $\scrP(2)=\{\{\{1\},\{2\}\},\{\{1,2\}\}\}$,
while we agree that $\scrP(0)=\{\emptyset\}$.
For any non-empty subset $B\subset\Z^+$ we write $\mathsf{m}_B:=\min_{j\in B}j$.
\begin{thm}\label{MOMENTTHM}
Let $k_1,k_2\in\Z_{\geq0}$ and fix numbers
$0<V_1\leq V_2\leq\cdots\leq V_{k_1}$
and $0<W_1\leq W_2\leq\cdots\leq W_{k_2}$.
Let $L$ be a random in $X_n$ with respect to $\mu$.
Then
\begin{align}\label{MOMENTTHMres}
\EE\biggl(\prod_{j=1}^{k_1}N_j(L)\prod_{j=1}^{k_2}\tN_j(L^*)\biggr)
\to \biggl(\sum_{P\in\scrP(k_1)}2^{k_1-\#P}\prod_{B\in P}V_{\mathsf{m}_B}\biggr)
\biggl(\sum_{P\in\scrP(k_2)}2^{k_2-\#P}\prod_{B\in P}W_{\mathsf{m}_B}\biggr)
\end{align}
as $n\to\infty$.
\end{thm}
\begin{remark}\label{k1ork2zeroREM}
In \eqref{MOMENTTHMres}, it should be noted that
an empty product equals $1$ by convention;
hence (using also $\scrP(0)=\{\emptyset\}$),
for $k_2=0$, %
\eqref{MOMENTTHMres} says
\begin{align}\label{MOMENTTHMresk2eq0}
\EE\biggl(\prod_{j=1}^{k_1}N_j(L)\biggr)
\to \sum_{P\in\scrP(k_1)}2^{k_1-\#P}\prod_{B\in P}V_{\mathsf{m}_B}.
\end{align}
This limit relation is already known from
\cite[Thm.\ 3, eq.\ (10) and Lemma 3]{aS2011o}.
Of course, by the $L\leftrightarrow L^*$ symmetry (see Remark \ref{DUALsymmrem}),
also the case $k_1=0$ is covered by the same reference.
\end{remark}

We now embark on the proof of Theorem \ref{MOMENTTHM}.
Because of Remark \ref{k1ork2zeroREM}, we may
from start assume that 
both $k_1$ and $k_2$ are positive.
In the following we will work with an arbitrary, fixed, 
dimension $n>k_1+k_2$.
Let $\rho_j$, $1\leq j\leq k_1$,
be the characteristic function of the open $n$-ball of volume $V_j$
centered at the origin, with the origin removed.
Then
\begin{align*}
N_j(L)=\sum_{\vecv\in L}\rho_j(\vecv).
\end{align*}
Similarly let $\trho_j$, $1\leq j\leq k_2$,
be the characteristic function of the open $n$-ball of volume $W_j$
centered at the origin, with the origin removed.
Then
\begin{align*}
\tN_j(L)=\sum_{\vecv\in L^*}\trho_j(\vecv).
\end{align*}
Also define $\rho:(\R^n)^{k_1}\times(\R^n)^{k_2}\to\R_{\geq0}$ through
\begin{align}\label{rhoCHOICE}
\rho(\vecv_1,\ldots,\vecv_{k_1},\vecw_1,\ldots,\vecw_{k_2})
:=\prod_{j=1}^{k_1}\rho_j(\vecv_j)\prod_{j=1}^{k_2}\trho_j(\vecw_j).
\end{align}
Then by construction,
\begin{align*}
\EE\biggl(\prod_{j=1}^{k_1}N_j(L)\prod_{j=1}^{k_2}\tN_j(L^*)\biggr)
=\int_{X_n}
\sum_{\vecv_1,\ldots,\vecv_{k_1}\in L}\sum_{\vecw_1,\ldots,\vecw_{k_2}\in L^*}
\rho(\vecv_1,\ldots,\vecv_{k_1},\vecw_1,\ldots,\vecw_{k_2})\,d\mu(L).
\end{align*}
By Theorem \ref{MAINTHEOREMabs2}, this equals
\begin{align}\label{MAINTHEOREMabs2resREP}
\sum_{m_1=1}^{k_1}\sum_{B_1\in A_{k_1,m_1}}\sum_{m_2=1}^{k_2}\sum_{B_2\in A_{k_2,m_2}}
\sum_{\beta\in M_{m_1,m_2}\!(\Z)} W(\beta)
\int_{S(\beta)}\rho\bigl(xB_1^{\mathsf{T}},yB_2^{\mathsf{T}}\bigr)\,d\eta_\beta(x,y).
\end{align}
Indeed, note that all the terms in the last line of \eqref{MAINTHEOREMabs2res}
vanish, since by definition we have $\rho_j(\bn)=0$ and $\trho_j(\bn)=0$ for all $j$.

\subsection{The main contribution}
\label{MAINTERMsec}

We will start by considering the sums over $B_1$ and $B_2$ in
\eqref{MAINTHEOREMabs2resREP}
restricted to certain subsets of $A_{k_1,m_1}$
and $A_{k_2,m_2}$, respectively;
it will turn out that these restricted sums give the main term contribution
as $n\to\infty$.

For any $k\geq m$, we let $\scrM_{k,m}$ be the set of all $k\times m$ matrices
which have all entries in $\{-1,0,1\}$,
exactly one non-zero entry in each row, and at least one non-zero entry in each column.
One verifies that $\scrM_{k,m}\subset M_{k,m}(\Z)^*$.
Let $\scrM_{k,m}'$ be the subset of those matrices 
$B=(b_{ij})\in\scrM_{k,m}$ which satisfy the following condition:
There exist some $1=\nu_1<\nu_2<\cdots<\nu_m\leq k$
such that for each $j\in\{1,\ldots,m\}$ we have
$b_{\nu_j,j}=1$ and $b_{ij}=0$ for all $i<\nu_j$.
Clearly the indices $\nu_1,\ldots,\nu_m$ are uniquely determined
for every $B\in\scrM_{k,m}'$, 
and we will denote these by $\nu_1(B),\ldots,\nu_m(B)$.

Let $\fS_m\subset\GL_m(\Z)$ be the group of $m\times m$ signed permutation matrices,
i.e.\ matrices which have all entries in $\{-1,0,1\}$ and exactly one non-zero entry in each row
and in each column.
One verifies that $\fS_m$ acts freely on $\scrM_{k,m}$
by multiplication from the right;
in particular each orbit for this action has size $\#\fS_m=2^mm!$.
Furthermore, whenever the relation
$B_1\gamma=B_2$ holds for some
$B_1,B_2\in\scrM_{k,m}$ and $\gamma\in\GL_m(\Z)$,
$\gamma$ must in fact belong to $\fS_m$.
Using these observations, one easily verifies that
$\scrM_{k,m}'$ contains exactly one representative from every
$\GL_m(\Z)$-orbit in $M_{k,m}(\Z)^*$ 
which intersects $\scrM_{k,m}$.
Hence we may, without loss of generality, from now on assume that 
the set of representatives $A_{k,m}$ has been chosen in such a way that
\begin{align*}
\scrM_{k,m}'\subset A_{k,m}.
\end{align*}

Let us now consider the contribution 
in \eqref{MAINTHEOREMabs2resREP}
from a fixed choice of
$m_1\in\{1,\ldots,k_1\}$,
$B_1\in\scrM_{k_1,m_1}'$,
$m_2\in\{1,\ldots,k_2\}$,
$B_2\in\scrM_{k_2,m_2}'$.
That is, we consider the following sum:
\begin{align}\label{MAINTERM1pre}
\sum_{\beta\in M_{m_1,m_2}\!(\Z)} W(\beta)
\int_{S(\beta)}\rho\bigl(xB_1^{\mathsf{T}},yB_2^{\mathsf{T}}\bigr)\,d\eta_\beta(x,y).
\end{align}
It is immediate from the definition in \eqref{Wdef} that,
for any $\beta\in M_{m_1,m_2}(\Z)$,
\begin{align}\label{Wbetatrivbounds}
\frac1{\zeta(n)\zeta(n-1)\cdots\zeta(n-m_1+1)}\leq W(\beta)\leq 
\frac{\zeta(n-m_2)\cdots\zeta(n-m_2-m_1+1)}{\zeta(n)\zeta(n-1)\cdots\zeta(n-m_1+1)},
\end{align}
with the second relation being an equality when $\beta=0$.
In particular we have
\begin{align}\label{Wbetatrivbounds2}
W(\beta)=1+O_m(2^{-n}),
\end{align}
where 
\begin{align*}
m:=m_1+m_2.
\end{align*}
Since $\rho\geq0$, it follows that \eqref{MAINTERM1pre}
equals
\begin{align}\label{MAINTERM1}
\bigl(1+O_m(2^{-n})\bigr)\sum_{\beta\in M_{m_1,m_2}\!(\Z)}
\int_{S(\beta)}\rho\bigl(xB_1^{\mathsf{T}},yB_2^{\mathsf{T}}\bigr)\,d\eta_\beta(x,y).
\end{align}
Let us here write $\rho$ as
\begin{align}\label{rhoCHOICEvar}
\rho(\vecv_1,\ldots,\vecv_{k_1},\vecw_1,\ldots,\vecw_{k_2})
=\rho'(\vecv_1,\ldots,\vecv_{k_1})\trho'(\vecw_1,\ldots,\vecw_{k_2})
\end{align}
with
\begin{align}\label{rhoptrhopDEF}
\rho'(\vecv_1,\ldots,\vecv_{k_1})
:=\prod_{j=1}^{k_1}\rho_j(\vecv_j)
\quad\text{and}\quad
\trho'(\vecw_1,\ldots,\vecw_{k_2})
:=\prod_{j=1}^{k_2}\trho_j(\vecw_j)
\end{align}
(cf.\ \eqref{rhoCHOICE}).
Then \eqref{MAINTERM1} can be expressed as follows, 
using also \eqref{NUbetaDEFnew}:
\begin{align}\notag
\bigl(1+O_m(2^{-n})\bigr)\sum_{\beta\in M_{m_1,m_2}\!(\Z)}
\int_{U_{m_1}}\int_{S(\beta)_x}\rho'\bigl(xB_1^{\mathsf{T}}\bigr)
\trho'\bigl(yB_2^{\mathsf{T}}\bigr)\,d\eta_{\beta,x}(y)\,\frac{dx}{\dd(x)^{m_2}}
\hspace{70pt}
\\\label{MAINTERM2}
=\bigl(1+O_m(2^{-n})\bigr)\int_{U_{m_1}}\rho'\bigl(xB_1^{\mathsf{T}}\bigr)
\sum_{\beta\in M_{m_1,m_2}\!(\Z)}
\int_{S(\beta)_x}
\trho'\bigl(yB_2^{\mathsf{T}}\bigr)\,d\eta_{\beta,x}(y)\,\frac{dx}{\dd(x)^{m_2}}.
\end{align}

Let the column vectors of $x$ be $\vecx_1,\ldots,\vecx_{m_1}$,
and let the column vectors of $y$
be $\vecy_1,\ldots,\vecy_{m_2}$.
Set $\nu_{j,1}:=\nu_j(B_1)$ and $\nu_{j,2}:=\nu_j(B_2)$.
It then follows from our definitions,
and in particular the assumption that
$V_1\leq\cdots\leq V_{k_1}$
and $W_1\leq\cdots\leq W_{k_2}$,
that
\begin{align}\label{MAINTERM3}
\rho'\bigl(xB_1^{\mathsf{T}}\bigr)=\prod_{j=1}^{m_1}\rho_{\nu_{j,1}}\bigl(\vecx_j\bigr)
\qquad
\text{and}
\qquad
\trho'\bigl(yB_2^{\mathsf{T}}\bigr)=\prod_{j=1}^{m_2}\trho_{\nu_{j,2}}\bigl(\vecy_j\bigr).
\end{align}
For $x\in U_{m_1}$ 
we set 
\begin{align}\label{bxDEF}
b_x:=x(x\transp x)^{-1}\in U_{m_1}.
\end{align}
Then for any $y\in M_{n,m_2}(\R)$
we note that $y\in S(\beta)_x'$
holds if and only if
$x\transp y=\beta=x\transp b_x\beta$,
viz.,
if and only if $y-b_x\beta\in(x^\perp)^{m_2}$,
where we recall that $x^\perp$ denotes the orthogonal complement in $\R^n$ of 
the column span of $x$.
In other words, $y\in S(\beta)_x'$ holds if and only if
$\vecy_j\in b_x\vecbeta_j+x^\perp$ 
for all $j\in\{1,\ldots,m_2\}$,
where $\vecbeta_1,\ldots,\vecbeta_{m_2}$ are the column vectors of $\beta$.
Hence
\begin{align}\label{MAINTERM6}
\int_{S(\beta)_x}
\trho'\bigl(yB_2^{\mathsf{T}}\bigr)\,d\eta_{\beta,x}(y)
=\prod_{j=1}^{m_2}\int_{b_x\vecbeta_j+x^\perp}\trho_{\nu_{j,2}}\bigl(\vecy_j\bigr)\,d\vecy_j,
\end{align}
where $d\vecy_j$ denotes the $(n-m_1)$-dimensional Lebesgue measure on the affine subspace
$b_x\vecbeta_j+x^\perp$.
But $\trho_{\nu_{j,2}}$ is the characteristic function of the open $n$-ball of radius
$R_j:=(W_{\nu_{j,2}}/\fV_n)^{1/n}$
centered at the origin, with the origin removed.
Hence the expression in \eqref{MAINTERM6}
equals $\prod_{j=1}^{m_2}f_j\bigl(\|b_x\vecbeta_j\|\bigr)$
where $f_j:\R_{\geq0}\to\R_{\geq0}$ is given by
\begin{align*}
f_j(r)=\begin{cases}
\fV_{n-m_1}\cdot(R_j^2-r^2)^{(n-m_1)/2}&\text{if }\: r\leq R_j
\\
0&\text{if }\: r\geq R_j.
\end{cases}
\end{align*}
It follows that, for every $x\in U_{m_1}$,
\begin{align}\label{MAINTERM5}
\sum_{\beta\in M_{m_1,m_2}\!(\Z)}
\int_{S(\beta)_x}
\trho'\bigl(yB_2^{\mathsf{T}}\bigr)\,d\eta_{\beta,x}(y)
=
\sum_{\beta\in M_{m_1,m_2}\!(\Z)}
\prod_{j=1}^{m_2}f_j\bigl(\|b_x\vecbeta_j\|\bigr)
=\prod_{j=1}^{m_2}\sum_{\vecbeta_j\in\Z^{m_1}}f_j\bigl(\|b_x\vecbeta_j\|\bigr).
\end{align}
In the last product, we will approximate each sum
$\sum_{\vecbeta_j\in\Z^{m_1}}f_j\bigl(\|b_x\vecbeta_j\|\bigr)$
by the corresponding \textit{integral},
\begin{align}\label{MAINTERM7}
\int_{\vecbeta\in\R^{m_1}}f_j\bigl(\|b_x\vecbeta\|\bigr)\,d\vecbeta
=\dd(b_x)^{-1}\int_{\vecz\in\R^{m_1}}f_j\bigl(\|\vecz\|\bigr)\,d\vecz
=\dd(b_x)^{-1}\fV_nR_j^n
=\dd(x)W_{\nu_{j,2}},
\end{align}
where in the last equality we used the fact that 
$\dd(b_x)^{-1}=\dd(x)$.
The following lemma gives a bound on the error in this approximation.
From now on we will write
$C:=\bigl[-\tfrac12,\tfrac12\bigr]^{m_1}$;
a closed unit cube in $\R^{m_1}$.
We also define:
\begin{align*}
\ell_x:=\sup_{\vecu\in C}\|b_x\vecu\|.
\end{align*}
\begin{lem}\label{MAINTERMLEM1}
Fix $\ve>0$ and $j\in\{1,\ldots,m_2\}$.
Assume that $n$ is so large that $n\geq m_1+3$
and $\ve^n\leq W_{\nu_{j,2}}\leq\ve^{-n}$.
Then for every $x\in U_{m_1}$ satisfying $\ell_x\leq1$, we have
\begin{align}\label{MAINTERMLEM1res}
\biggl|\sum_{\vecbeta\in\Z^{m_1}}
f_j\bigl(\|b_x\vecbeta\|\bigr)
-\dd(x)W_{\nu_{j,2}}\biggr|
\ll_{m_1,\ve}
W_{\nu_{j,2}}\dd(x)\ell_x.
\end{align}
\end{lem}
\begin{proof}
We split the integral in \eqref{MAINTERM7} by tesselating $\R^{m_1}$ 
by translates of the unit cube $C$;
this gives:
\begin{align}\label{MAINTERMLEM1pf1}
\dd(x)W_{\nu_{j,2}}
=\sum_{\vecbeta\in\Z^{m_1}}\int_{\vecbeta+C}f_j\bigl(\|b_x\vecu\|\bigr)\,d\vecu.
\end{align}
Using this equality, and the triangle inequality,
it follows that the left-hand side of \eqref{MAINTERMLEM1res} is
bounded above by
\begin{align*}
\sum_{\vecbeta\in\Z^{m_1}}\int_{C}\Bigl|f_j\bigl(\|b_x\vecbeta\|\bigr)
-f_j\bigl(\|b_x(\vecbeta+\vecu)\|\bigr)\Bigr|\,d\vecu.
\end{align*}
Since $n\geq m_1+3$, the function $f_j(r)$ is $C^1$,
and we have
\begin{align}\label{MAINTERMLEM1pf2}
\bigl|f_j(a)-f_j(b)\bigr|\leq (b-a)\sup_{s\in[a,b]}|f_j'(s)|
\qquad\text{for all $0\leq a\leq b$.}
\end{align}
Hence, recalling the definition of $\ell_x$,
it follows that the integrand in \eqref{MAINTERMLEM1pf1} is everywhere
bounded above by $\ell_x\cdot F_j\bigl(\|b_x\vecbeta\|\bigr)$,
where
\begin{align*}
F_j(r):=\sup\Bigl\{|f_j'(s)|\col s\in\bigl[\max(0,r-\ell_x),r+\ell_x\bigr]\Bigr\}.
\end{align*}
Hence \eqref{MAINTERMLEM1pf2} 
(and thus also the left-hand side of \eqref{MAINTERMLEM1res})
is bounded above by
\begin{align}\label{MAINTERMLEM1pf3} 
\ell_x\sum_{\vecbeta\in\Z^{m_1}}F_j\bigl(\|b_x\vecbeta\|\bigr)
=-\ell_x\int_0^\infty F_j'(r)A(r)\,dr
\end{align}
where
\begin{align*}
A(r):=\#\{\vecbeta\in\Z^{m_1}\col \|b_x\vecbeta\|\leq r\}.
\end{align*}
(The equality in \eqref{MAINTERMLEM1pf3}
holds by integration by parts;
note that $F_j(r)=0$ for all $r\geq R_j+\ell_x$,
so that the range of integration can just as well be taken to be $[0,R_j+\ell_x]$.)

We compute
\begin{align*}
\bigl|f_j'(s)\bigr|=\fV_{n-m_1}(n-m_1)\bigl(R_j^2-s^2\bigr)^{\frac{n-m_1}2-1}s
\qquad(0\leq s\leq R_j).
\end{align*}
and this function is verified to be increasing for $0\leq s\leq\kappa R_j$ %
and decreasing for $\kappa R_j\leq s\leq R_j$, %
where
\begin{align*}
\kappa:=(n-m_1-1)^{-\frac12}.
\end{align*}
(Note that $0<\kappa\leq 2^{-\frac12}$.)
Hence the function $F_j(r)$ is increasing for
$0\leq r\leq\kappa R_j+\ell_x$
(in fact it is constant on the interval
$\max(0,\kappa R_j-\ell_x)\leq r\leq\kappa R_j+\ell_x$),
and decreasing for $r\geq\kappa R_j+\ell_x$.
Note also that for $r\geq\kappa R_j+\ell_x$ we have
$F_j(r)=|f_j'(r-\ell_x)|$,
and thus
for all $r\in \bigl[\kappa R_j,R_j\bigr)$ we have:
\begin{align*}
F_j'(r+\ell_x)=\frac d{dr}|f_j'(r)|
=\fV_{n-m_1}(n-m_1)\bigl(R_j^2-r^2\bigr)^{\frac{n-m_1}2-2}
\Bigl(R_j^2-(n-m_1-1)r^2\Bigr)
\end{align*}
It follows that \eqref{MAINTERMLEM1pf3} 
(and thus also the left-hand side of \eqref{MAINTERMLEM1res})
is bounded above by
\begin{align}\notag
-\ell_x\int_{\kappa R_j+\ell_x}^{R_j+\ell_x} F_j'(r)A(r)\,dr
\hspace{280pt}
\\\label{MAINTERMLEM1pf4}
=\fV_{n-m_1}(n-m_1)\ell_x\int_{\kappa R_j}^{R_j}
\bigl(R_j^2-r^2\bigr)^{\frac{n-m_1}2-2}
\Bigl((n-m_1-1)r^2-R_j^2\Bigr)\,A(r+\ell_x)\,dr.
\end{align}

In order to bound $A(r)$, note that for every 
$\vecbeta\in\Z^{m_1}$ with $\|b_x\vecbeta\|\leq r$,
the parallelogram $b_x(\vecbeta+C)$ is contained in
$\scrB_{r+\ell_x}^n\cap V_x$ (recall that $V_x=x\R^{m_1}=\Span_{\R}\{\vecx_1,\ldots,\vecx_{m_1}\}$). 
Furthermore these
parallelograms are pairwise disjoint,
and each of them has volume $\dd(b_x)=\dd(x)^{-1}$,
whereas the ball $\scrB_{r+\ell_x}^n\cap V_x$ has volume $\fV_{m_1}(r+\ell_x)^{m_1}$.
(Here ``volume'' refers to the $m_1$-dimensional Lebesgue measure on $V_x$.)
Hence:
\begin{align*}
A(r)\leq \fV_{m_1}\,\dd(x)\,(r+\ell_x)^{m_1},\qquad\forall r\geq0.
\end{align*}
Using this bound in \eqref{MAINTERMLEM1pf4},
we conclude that the left-hand side of \eqref{MAINTERMLEM1res}
is bounded above by
\begin{align}\notag
\fV_{n-m_1}\fV_{m_1}(n-m_1)\dd(x)\ell_x\int_{\kappa R_j}^{R_j}
\bigl(R_j^2-r^2\bigr)^{\frac{n-m_1}2-2}
\Bigl((n-m_1-1)r^2-R_j^2\Bigr)\,(r+2\ell_x)^{m_1}\,dr.
\\\label{MAINTERMLEM1pf5}
\ll_{m_1}\fV_{n-m_1}n^2\dd(x)\ell_x\int_{\kappa R_j}^{R_j}
\bigl(R_j^2-r^2\bigr)^{\frac{n-m_1}2-2}
r^2\,(r+2\ell_x)^{m_1}\,dr.  %
\end{align}
Here the integral is
\begin{align}\notag
\ll_{m_1}
\sum_{a=0}^{m_1}\ell_x^{m_1-a}\int_0^{R_j}
\bigl(R_j^2-r^2\bigr)^{\frac{n-m_1}2-2}r^{2+a}\,dr
=\frac12
\sum_{a=0}^{m_1}\ell_x^{m_1-a}R_j^{n-m_1-1+a}\frac{\Gamma\bigl(\frac{n-m_1}2-1\bigr)\Gamma\bigl(\frac{a+3}2\bigr)}
{\Gamma\bigl(\frac{n-m_1+a+1}2\bigr)}
\\\label{MAINTERMLEM1pf6}
\ll_{m_1}\sum_{a=0}^{m_1}\ell_x^{m_1-a}R_j^{n-m_1-1+a} n^{-\frac{a+3}2}
\ll_{m_1,\ve} R_j^n\cdot n^{-\frac{m_1}2-2},
\end{align}
where in the last step we 
used $\ell_x\leq1$,
and we also used
$R_j=(W_{\nu_{j,2}}/\fV_n)^{1/n}$
and $\ve^n\leq W_{\nu_{j,2}}\leq\ve^{-n}$
to see that $R_j\asymp_{\ve}\sqrt n$.
Combining \eqref{MAINTERMLEM1pf5} and \eqref{MAINTERMLEM1pf6},
and using $\fV_{n-m_1}\asymp_{m_1}\fV_n n^{m_1/2}$
and $\fV_n R_j^n=W_{\nu_{j,2}}$,
we obtain the first bound in \eqref{MAINTERMLEM1res}.
\end{proof}
We complement Lemma \ref{MAINTERMLEM1} with a cruder bound which is valid regardless of the size of $\ell_x$.
\begin{lem}\label{MAINTERMLEM2}
Fix $\ve>0$ and $j\in\{1,\ldots,m_2\}$.
Assume that $n$ is so large that $n\geq m_1+3$
and $\ve^n\leq V_{k_1},W_{\nu_{j,2}}\leq\ve^{-n}$.
Then for every $x\in U_{m_1}$
with $\rho'(xB_1^{\mathsf{T}})=1$, we have
\begin{align}\label{MAINTERMLEM2res}
\sum_{\vecbeta\in\Z^{m_1}}
f_j\bigl(\|b_x\vecbeta\|\bigr)
\ll_{m_1,\ve} W_{\nu_{j,2}}\, n^{m_1}
\quad\text{and}\quad
\dd(x)W_{\nu_{j,2}}
\ll_{m_1,\ve}
W_{\nu_{j,2}}\, n^{m_1}.
\end{align}
\end{lem}
\begin{proof}
It follows from 
$\rho'(xB_1^{\mathsf{T}})=1$
that %
$\|\vecx_j\|<(V_{k_1}/\fV_n)^{1/n}\ll_{\ve} \sqrt n$ for each $j$.
Hence $\dd(x)\ll_{m_1,\ve} n^{m_1}$,
and thus the second bound in \eqref{MAINTERMLEM2res} holds.

In order to prove the first bound
in \eqref{MAINTERMLEM2res},
we start by noticing
\begin{align}\label{MAINTERMLEM2pf1}
\sum_{\vecbeta\in\Z^{m_1}}
f_j\bigl(\|b_x\vecbeta\|\bigr)
\leq \fV_{n-m_1}R_j^{n-m_1}\cdot
\#\bigl\{\vecbeta\in\Z^{m_1}\col \|b_x\vecbeta\|<R_j\bigr\},
\end{align}
since $f_j(r)\leq\fV_{n-m_1}R_j^{n-m_1}$
for all $r\geq0$.
Now note that for all $\vecbeta\in\Z^{m_1}\setminus\{\bn\}$
we have $\|x\transp b_x\vecbeta\|=\|\vecbeta\|\geq1$,
and combining this with the fact that 
$\|\vecx_j\|\ll_{\ve}\sqrt n$ for each $j\in\{1,\ldots,m_1\}$,
it follows that $\|b_x\vecbeta\|\gg_{m_1,\ve} n^{-\frac12}$.
It follows that the distance between any two vectors 
$b_x\vecbeta$ and $b_x\vecbeta'$ for $\vecbeta\neq\vecbeta'\in\Z^{m_1}$
is $\gg_{m_1,\ve} n^{-\frac12}$,
and hence there exists a number $r\gg_{m_1,\ve} n^{-\frac12}$
such that the balls $\scrB_r^n+b_x\vecbeta$ are pairwise disjoint as $\vecbeta$ runs through $\Z^{m_1}$.
We may also require that $r\leq n^{-\frac12}$
(indeed otherwise replace $r$ by $n^{-\frac12}$).
Then it follows that for every 
$\vecbeta\in\Z^{m_1}$ with $\|b_x\vecbeta\|<R_j$,
the ball $\scrB_r^n+b_x\vecbeta$
is contained in $\scrB_{R_j+n^{-1/2}}^n$.
Intersecting these balls with $V_x$ and considering the volume,
it follows that
\begin{align}\label{MAINTERMLEM2pf2}
\#\bigl\{\vecbeta\in\Z^{m_1}\col \|b_x\vecbeta\|<R_j\bigr\}
\leq\frac{\fV_{m_1}(R_j+n^{-1/2})^{m_1}}{\fV_{m_1}r^{m_1}}
\ll_{m_1,\ve} R_j^{m_1}r^{-m_1}
\ll_{m_1,\ve} R_j^{m_1}n^{\frac{m_1}2}.
\end{align}
(Here we used $\ve^n\leq W_{\nu_{j,2}}\leq\ve^{-n}$
to conclude that $n^{-1/2}\leq n^{1/2}\ll_{\ve} R_j$.)
Combining \eqref{MAINTERMLEM2pf1}
and \eqref{MAINTERMLEM2pf2},
and using $\fV_{n-m_1}\asymp_{m_1}\fV_n n^{\frac{m_1}2}$,
we obtain the first bound in \eqref{MAINTERMLEM2res}.
\end{proof}

We will now apply 
Lemma \ref{MAINTERMLEM1}
and 
Lemma \ref{MAINTERMLEM2}
to estimate the expression in \eqref{MAINTERM5}.
Let us pick $\ve=\frac12$ in both the lemmas;
thus in order for the lemmas to apply,
from now on we assume that $n$ is so large that
$n\geq m_1+3$ and so that all the $V_i$s and all the $W_i$s
lie in the interval $[2^{-n},2^n]$.
Recall that we also assume $n>k_1+k_2\geq m_1+m_2$.
Using Lemma \ref{MAINTERMLEM1},
we conclude that for 
every $x\in U_{m_1}$ satisfying $\ell_x\leq1$,
the expression in \eqref{MAINTERM5} equals
\begin{align*}
\biggl(\prod_{j=1}^{m_2}W_{\nu_{j,2}}\biggr)\dd(x)^{m_2}\Bigl(1+O_{m}(\ell_x)\Bigr).
\end{align*}

Using Lemma \ref{MAINTERMLEM2},
we conclude that for 
every $x\in U_{m_1}$ satisfying $\rho'(xB_1^{\mathsf{T}})=1$,
the expression in \eqref{MAINTERM5} equals
\begin{align}\label{MAINTERM5a}
\biggl(\prod_{j=1}^{m_2}W_{\nu_{j,2}}\biggr)\dd(x)^{m_2}
\prod_{j=1}^{m_2}\Bigl(1+O_{m_1}\Bigl(\frac{n^{m_1}}{\dd(x)}\Bigr)\Bigr)
=\biggl(\prod_{j=1}^{m_2}W_{\nu_{j,2}}\biggr)\dd(x)^{m_2}
\Bigl(1+O_{m}\Bigl(\frac{n^{m_1m_2}}{\dd(x)^{m_2}}\Bigr)\Bigr).
\end{align}

We will use the first of these estimates when $\ell_x\leq c$ and the second when $\ell_x>c$,
where $c$ is a parameter in the interval $(0,1]$
which we will choose later.
Let $\chi_c$ be the characteristic function of $\ell_x\leq c$
and let $\tchi_c=1-\chi_c$ be the characteristic function of $\ell_x>c$.
Then we conclude that 
\eqref{MAINTERM2} (and thus also \eqref{MAINTERM1pre})
equals
\begin{align}\notag
\bigl(1+O_m(2^{-n})\bigr)\biggl(\prod_{j=1}^{m_2}W_{\nu_{j,2}}\biggr)
\Biggl(\int_{U_{m_1}}\rho'(xB_1^{\mathsf{T}})\,dx
+O_{m}\biggl(\int_{U_{m_1}}\rho'(xB_1^{\mathsf{T}})\,\chi_c(x)\,\ell_x\,dx
\hspace{40pt}
\\\label{MAINTERM8}
+n^{m_1m_2}\int_{U_{m_1}}\rho'(xB_1^{\mathsf{T}})\,\tchi_c(x)\,\frac{dx}{\dd(x)^{m_2}}\biggr)\Biggr)
\end{align}
Here we have, by \eqref{MAINTERM3},
\begin{align*}
\int_{U_{m_1}}\rho'(xB_1^{\mathsf{T}})\,dx=\prod_{j=1}^{m_1}V_{\nu_{j,1}}.
\end{align*}
Furthermore, the first error term is
\begin{align*}
\int_{U_{m_1}}\rho'(xB_1^{\mathsf{T}})\,\chi_c(x)\,\ell_x\,dx
\leq c\int_{U_{m_1}}\rho'(xB_1^{\mathsf{T}})\,dx
=c\prod_{j=1}^{m_1}V_{\nu_{j,1}}.
\end{align*}

In order to bound the second error term in \eqref{MAINTERM8},
we apply
\cite[Lemma 5.2]{aSaS2016};
combined with \eqref{MAINTERM3} this gives that
\begin{align}\notag
\int_{U_{m_1}}&\rho'(xB_1^{\mathsf{T}})\,\tchi_c(x)\,\frac{dx}{\dd(x)^{m_2}}
\\\label{MAINTERM9}
&=
\frac{\prod_{j=n-m_1+1}^{n}(j\fV_j)}{\prod_{j=1}^{m_1}(j\fV_j)}
\int_{U_{m_1}\cap(\R^{m_1}\times\{\bn\})^{m_1}}\rho'(xB_1^{\mathsf{T}})\,\tchi_c(x)\,\dd(x)^{n-m_1-m_2}\,dx,
\end{align}
where in the second line, 
$\R^{m_1}\times\{\bn\}$ denotes the subspace
$\{\vecx\in\R^n\col x_{m_1+1}=\cdots=x_n=0\}$ of $\R^n$,
and $dx$ denotes the $m_1^2$-dimensional Lebesgue measure on
$(\R^{m_1}\times\{\bn\})^{m_1}$
(whereas in the first line it is the $m_1n$-dimensional Lebesgue measure on $(\R^n)^{m_1}$).
We will now use the following lemma:
\begin{lem}\label{MAINTERMLEM3}
For any $x\in U_{m_1}$ with
$\rho'(xB_1^{\mathsf{T}})\tchi_c(x)=1$,
we have
$\dd(x)\ll_{m_1} c^{-1}n^{\frac{m_1-1}2}$.
\end{lem}
\begin{proof}
Assume that 
$x\in U_{m_1}$ and
$\rho'(xB_1^{\mathsf{T}})\tchi_c(x)=1$.
Then
\begin{align*}
c<\ell_x\leq\frac12\sum_{j=1}^{m_1}\|b_x\vece_j\|,
\end{align*}
and hence there exists some $j\in\{1,\ldots,m_1\}$
for which $\|b_x\vece_j\|>(2/m_1)c$.
We have $xb_x\vece_j=\vece_j$,
which implies that
$\vecx_j\cdot b_x\vece_j=1$
while $b_x\vece_j$ is
orthogonal to the subspace
\begin{align*}
V_{x,j}:=\Span_{\R}\{\vecx_i\col i\in\{1,\ldots,m_1\}\setminus\{j\}\}.
\end{align*}
It follows that the orthogonal projection of $\vecx_j$ onto $V_{x,j}^\perp$
has length
\begin{align*}
\vecx_j\cdot\frac{b_x\vece_j}{\|b_x\vece_j\|}=\|b_x\vece_j\|^{-1}.
\end{align*}
Hence $\dd(x)=\|b_x\vece_j\|^{-1}\dd'$ where $\dd'$ is the $(m_1-1)$-dimensional volume of the
parallelotope in $\R^n$ spanned by the vectors $\vecx_1,\ldots,\vecx_{m_1}$ except $\vecx_j$.
Clearly $\dd'\leq\prod_{\substack{i=1\\(i\neq j)}}^{m_1}\|\vecx_i\|$,
and $\rho'(xB_1^{\mathsf{T}})\tchi_c(x)=1$
implies that $\|\vecx_i\|<(V_{k_1}/\fV_n)^{1/n}
\leq(2^n/\fV_n)^{1/n}\ll n^{1/2}$ for each $i$.
Hence:
\begin{align*}
\dd(x)\ll_{m_1}\|b_x\vece_j\|^{-1}n^{\frac{m_1-1}2}
\ll_{m_1}c^{-1}n^{\frac{m_1-1}2}.
\end{align*}
\end{proof}
Let $K=K(m_1)$ be the implied constant in the bound in Lemma \ref{MAINTERMLEM3}.
It follows that the expression in 
\eqref{MAINTERM9} is
\begin{align*}
&\leq \frac{\prod_{j=n-m_1+1}^{n}(j\fV_j)}{\prod_{j=1}^{m_1}(j\fV_j)}
\int_{U_{m_1}\cap(\R^{m_1}\times\{\bn\})^{m_1}}\rho'(xB_1^{\mathsf{T}})
\cdot\bigl(Kc^{-1}n^{\frac{m_1-1}2}\bigr)^{n-m_1-m_2}\,dx
\\
&=\frac{\prod_{j=n-m_1+1}^{n}(j\fV_j)}{\prod_{j=1}^{m_1}(j\fV_j)}
\bigl(Kc^{-1}n^{\frac{m_1-1}2}\bigr)^{n-m_1-m_2}
\prod_{j=1}^{m_1}\biggl(\fV_{m_1}\Bigl(\frac{V_{\nu_{j,1}}}{\fV_n}\Bigr)^{m_1/n}\biggr)
\\
&\ll_{m} 
\Biggl(\prod_{j=0}^{m_1-1}\biggl(\Bigl(\frac{2\pi e}n\Bigr)^{\frac{n}2} n^{\frac{j+1}2}\biggr)\Biggr)
K^nc^{m-n}n^{\frac{m_1-1}2(n-m)}
\prod_{j=1}^{m_1} n^{m_1/2}
\\
&=\biggl(\frac{(2\pi e)^{m_1/2}K}{c\sqrt n}\biggr)^{n} n^{\frac{m_1^2}4+\frac{m_1}4-\frac{m_1m_2}2+\frac m2} c^m.
\end{align*}
Using this bound,
together with the fact that
$\prod_{j=1}^{m_1}V_{\nu_{j,1}}^{-1}\leq 2^{m_1n}$,
we conclude that
\eqref{MAINTERM8} (and thus also \eqref{MAINTERM1pre})
equals
\begin{align}\notag
\bigl(1+O_m(2^{-n})\bigr)
\biggl(\prod_{j=1}^{m_1}V_{\nu_{j,1}}\biggr)
\biggl(\prod_{j=1}^{m_2}W_{\nu_{j,2}}\biggr)
\Biggl(1+O_m\biggl(c+
\biggl(\frac{2^{m_1}(2\pi e)^{m_1/2}K}{c\sqrt n}\biggr)^{n} n^{\frac{m_1^2}4+\frac{m_1}4+\frac{m_1m_2}2+\frac m2} c^m
\biggr)\Biggr).
\end{align}
Choosing now $c:=2^{m_1+2}(2\pi e)^{m_1/2}Kn^{-\frac12}$, 
this becomes
\begin{align}\label{MAINTERM10}
\biggl(\prod_{j=1}^{m_1}V_{\nu_{j,1}}\biggr)
\biggl(\prod_{j=1}^{m_2}W_{\nu_{j,2}}\biggr)
\Bigl(1+O_m\bigl(n^{-\frac12}\bigr)\Bigr).
\end{align}

To sum up,
we have proved that the contribution 
to \eqref{MAINTHEOREMabs2resREP}
from any fixed choice of
$m_1\in\{1,\ldots,k_1\}$,
$B_1\in\scrM_{k_1,m_1}'$,
$m_2\in\{1,\ldots,k_2\}$,
$B_2\in\scrM_{k_2,m_2}'$
equals \eqref{MAINTERM10}
whenever $n$ is sufficiently large.
Recall here that 
$\nu_{j,1}:=\nu_j(B_1)$ and $\nu_{j,2}:=\nu_j(B_2)$.
Hence, since each set $\scrM_{k,m}'$ is finite,
it follows that 
\begin{align*}
\sum_{m_1=1}^{k_1}\sum_{B_1\in \scrM'_{k_1,m_1}}\sum_{m_2=1}^{k_2}\sum_{B_2\in \scrM'_{k_2,m_2}}
\sum_{\beta\in M_{m_1,m_2}\!(\Z)} W(\beta)
\int_{S(\beta)}\rho\bigl(xB_1^{\mathsf{T}},yB_2^{\mathsf{T}}\bigr)\,d\eta_\beta(x,y)
\hspace{50pt}
\\
\to\biggl(\sum_{m_1=1}^{k_1}\sum_{B_1\in \scrM'_{k_1,m_1}}
\prod_{j=1}^{m_1}V_{\nu_j(B_1)}\biggr)
\cdot\biggl(
\sum_{m_2=1}^{k_2}\sum_{B_2\in \scrM'_{k_2,m_2}}
\prod_{j=1}^{m_2}W_{\nu_j(B_2)}\biggr),
\qquad\text{as }\: n\to\infty.
\end{align*}

Finally we note that for any $k\in\Z^+$,
the matrices in %
$\sqcup_{m=1}^k\scrM_{k,m}'$
are exactly the transposes of the matrices considered in
\cite[Thm.\ 3]{aS2011o},
with the ``$\nu$'' being the same as ours.
Hence by \cite[Lemma 3]{aS2011o},
for any given $1=\nu_1<\cdots<\nu_m\leq k$,
the number of matrices $B\in\scrM_{k,m}'$ having
$\nu_j(B)=\nu_j$ for all $j$ equals 
$2^{k-m}$ times\footnote{This factor $2^{k-m}$ comes from the number of ways to choose the \textit{signs} of the matrix entries.}
the number of partitions 
$P=\{B_1,\ldots,B_{\#P}\}$ 
in $\scrP(k)$ satisfying $\#P=m$ and $\{\mathsf{m}_{B_1},\ldots,\mathsf{m}_{B_m}\}=\{\nu_1,\ldots,\nu_m\}$.
Therefore the previous limit relation can be re-expressed as
\begin{align}\notag
\sum_{m_1=1}^{k_1}\sum_{B_1\in \scrM'_{k_1,m_1}}\sum_{m_2=1}^{k_2}\sum_{B_2\in \scrM'_{k_2,m_2}}
\sum_{\beta\in M_{m_1,m_2}\!(\Z)} W(\beta)
\int_{S(\beta)}\rho\bigl(xB_1^{\mathsf{T}},yB_2^{\mathsf{T}}\bigr)\,d\eta_\beta(x,y)
\hspace{50pt}
\\\label{MAINTERM11}
\to
\biggl(\sum_{P\in\scrP(k_1)}2^{k_1-\#P}\prod_{B\in P}V_{\mathsf{m}_B}\biggr)
\biggl(\sum_{P\in\scrP(k_2)}2^{k_2-\#P}\prod_{B\in P}W_{\mathsf{m}_B}\biggr),
\qquad\text{as }\: n\to\infty.
\end{align}

\subsection{Bounding the remaining terms}

In view of \eqref{MAINTERM11},
in order to complete the proof of Theorem \ref{MOMENTTHM},
it remains to prove that the total contribution from
all terms in \eqref{MAINTHEOREMabs2resREP}
with $B_1\notin\scrM_{k_1,m_1}'$ or $B_2\notin\scrM_{k_2,m_2}'$
tends to zero as $n\to\infty$.

To start with, we will prove a bound on the contribution
in \eqref{MAINTHEOREMabs2resREP}
from any given
$m_1\in\{1,\ldots,k_1\}$,
$B_1\in A_{k_1,m_1}$,
$m_2\in\{1,\ldots,k_2\}$,
$B_2\in A_{k_2,m_2}$,
that is, on the expression
\begin{align}\label{MAINTERM12}
\sum_{\beta\in M_{m_1,m_2}\!(\Z)} W(\beta)
\int_{S(\beta)}\rho\bigl(xB_1^{\mathsf{T}},yB_2^{\mathsf{T}}\bigr)\,d\eta_\beta(x,y).
\end{align}
Throughout our discussion, we will assume that $n$ is so large that
\begin{align}\label{MAINTERM13}
\max(V_{k_1},W_{k_2})\leq 2^n.
\end{align}
Recall that we write $m:=m_1+m_2$.
Let us now also set $k:=k_1+k_2$.
\begin{lem}\label{MAINTERMLEM4}
The number of $\beta\in M_{m_1,m_2}\!(\Z)$ for which
$\int_{S(\beta)}\rho\bigl(xB_1^{\mathsf{T}},yB_2^{\mathsf{T}}\bigr)\,d\eta_\beta(x,y)>0$,
is $\ll_k n^{k_1k_2}$.
\end{lem}
\begin{proof}
If $\beta\in M_{m_1,m_2}\!(\Z)$ satisfies
$\int_{S(\beta)}\rho(xB_1^{\mathsf{T}},yB_2^{\mathsf{T}})\,d\eta_\beta(x,y)>0$
then there exists some $\langle x,y\rangle\in S(\beta)$
for which $\rho(xB_1^{\mathsf{T}},yB_2^{\mathsf{T}})>0$.
By the definition of $\rho$,
and using \eqref{MAINTERM13},
this implies that all the column vectors of
$xB_1^{\mathsf{T}}$ and of $yB_2^{\mathsf{T}}$
have lengths %
$\leq(2^n/\fV_n)^{1/n}\ll\sqrt n$.
But $\langle x,y\rangle\in S(\beta)$ means that
$x\transp y=\beta$, and so
$B_1\beta B_2^{\mathsf{T}}=B_1x\transp y B_2^{\mathsf{T}}
=(xB_1^{\mathsf{T}})\transp(y B_2^{\mathsf{T}})$.
This means that every matrix entry
of $B_1\beta B_2^{\mathsf{T}}$
is equals the scalar product of a column vector of 
$xB_1^{\mathsf{T}}$ and a column vector of $yB_2^{\mathsf{T}}$.
Hence, by the Cauchy--Schwarz inequality,
all the matrix entries of $B_1\beta B_2^{\mathsf{T}}$ are $\ll n$.
It follows that the number of possibilities for the matrix
$B_1\beta B_2^{\mathsf{T}}\in M_{k_1,k_2}(\Z)$ is $\ll_k n^{k_1k_2}$.
The proof of the lemma is now concluded by noticing that 
the map $\beta\mapsto B_1\beta B_2^{\mathsf{T}}$ is injective,
since $\dd(B_1)>0$ and $\dd(B_2)>0$.
\end{proof}
\begin{lem}\label{MAINTERMLEM5}
For every $\beta\in M_{m_1,m_2}(\R)$, we have
\begin{align}\label{MAINTERMLEM5res}
&\int_{S(\beta)}\rho\bigl(xB_1^{\mathsf{T}},yB_2^{\mathsf{T}}\bigr)\,d\eta_\beta(x,y)
\leq \int_{S(0)}\rho\bigl(xB_1^{\mathsf{T}},yB_2^{\mathsf{T}}\bigr)\,d\eta_0(x,y)
\\\notag
&=\frac{\prod_{j=n-m_1+1}^n (j\fV_j)}{\prod_{j=n-m_1-m_2+1}^{n-m_2} (j\fV_j)}
\int_{(\R^{n-m_2}\times\{\bn\})^{m_1}}\rho'(xB_1^{\mathsf{T}})\,dx
\int_{(\R^{n-m_1}\times\{\bn\})^{m_2}}\trho'(xB_2^{\mathsf{T}})\,dx,
\end{align}
where in the last two integrals, 
$\rho'$ and $\trho'$ are as in \eqref{rhoptrhopDEF},
and for any $0<\ell<n$, 
$\R^\ell\times\{\bn\}$
denotes the subspace
$\{\vecx\in\R^n\col x_{\ell+1}=\cdots=x_n=0\}$ of $\R^n$,
and ``$dx$'' stands for the natural Lebesgue measure in each case.
\end{lem}
\begin{proof}
By \eqref{NUbetaDEFnew} we have
\begin{align*}
\int_{S(\beta)}\rho\bigl(xB_1^{\mathsf{T}},yB_2^{\mathsf{T}}\bigr)\,d\eta_\beta(x,y)
=\int_{U_{m_1}}\rho'(xB_1^{\mathsf{T}})
\int_{S(\beta)_x}
\trho'\bigl(yB_2^{\mathsf{T}}\bigr)\,d\eta_{\beta,x}(y)\,\frac{dx}{\dd(x)^{m_2}}.
\end{align*}
Given $x\in U_{m_1}$
we set
$b_x:=x(x\transp x)^{-1}\in U_{m_1}$
as in \eqref{bxDEF}, 
and then we have $S(\beta)_x'=b_x\beta+(x^\perp)^{k_2}$,
and so
\begin{align*}
\int_{S(\beta)_x}
\trho'\bigl(yB_2^{\mathsf{T}}\bigr)\,d\eta_{\beta,x}(y)
=\int_{(x^\perp)^{k_2}}\trho'\bigl((b_x\beta+y')B_2^{\mathsf{T}}\bigr)\,dy',
\end{align*}
where $dy'$ is the $m_2(n-m_1)$-dimensional Lebesgue measure on $(x^\perp)^{k_2}$.
By \eqref{rhoptrhopDEF}, for any $w\in M_{n,k_2}(\R)$ we have
$\trho'(w)=1$ if and only if the $j$th column vector of $w$ is non-zero and has length
$<(W_j/\fV_n)^{1/n}$, for each $j\in\{1,\ldots,k_2\}$;
otherwise $\trho'(w)=0$.
However, for every $y'\in(x^\perp)^{k_2}$
we have $x\transp y'=0$ and hence
$(b_x\beta)\transp y' %
=0$, which means that 
every column vector of $b_x\beta$ is orthogonal against every column vector of $y'$.
Hence also
every column vector of $b_x\beta B_2\transp$ is orthogonal against every column vector of $y'B_2\transp$,
and therefore, for each $j\in\{1,\ldots,k_2\}$,
the $j$th column vector of $(b_x\beta+y')B_2\transp$ is at least as long as the $j$th column vector of $y'B_2\transp$.
It follows that for every $y'\in(x^\perp)^{k_2}$
we have $\trho'\bigl((b_x\beta+y')B_2^{\mathsf{T}}\bigr)
\leq \trho'\bigl(y'B_2^{\mathsf{T}}\bigr)$.
Hence:
\begin{align*}
\int_{S(\beta)_x}
\trho'\bigl(yB_2^{\mathsf{T}}\bigr)\,d\eta_{\beta,x}(y)
\leq
\int_{(x^\perp)^{k_2}}\trho'\bigl(y'B_2^{\mathsf{T}}\bigr)\,dy'
=\int_{S(0)_x}
\trho'\bigl(yB_2^{\mathsf{T}}\bigr)\,d\eta_{0,x}(y).
\end{align*}
Since this holds for every $x\in U_{m_1}$,
we conclude that the first relation (the inequality) in \eqref{MAINTERMLEM5res} holds.

Next, to prove the equality in \eqref{MAINTERMLEM5res},
we first note that, using the fact that $\trho'$ is rotationally invariant
in the sense that $\trho'(\kappa z)=\trho'(z)$ for all $\kappa\in O(n)$ and $z\in M_{n,k_2}$,
we have
\begin{align*}
\int_{(x^\perp)^{k_2}}\trho'\bigl(y'B_2^{\mathsf{T}}\bigr)\,dy'
=\int_{(\R^{n-m_1})^{m_2}}\trho'\bigl(y'B_2^{\mathsf{T}}\bigr)\,dy'
\end{align*}
for each $x\in U_{m_1}$.
Note that the latter integral is independent of $x$.
Hence we get
\begin{align*}
\int_{S(0)}\rho\bigl(xB_1^{\mathsf{T}},yB_2^{\mathsf{T}}\bigr)\,d\eta_0(x,y)
=\int_{U_{m_1}}\rho'(xB_1^{\mathsf{T}})\,\frac{dx}{\dd(x)^{m_2}}
\int_{(\R^{n-m_1})^{m_2}}\trho'\bigl(y'B_2^{\mathsf{T}}\bigr)\,dy'.
\end{align*}
Finally, by using the corresponding rotational invariance of $\rho'$
and applying 
\cite[Lemma 5.2]{aSaS2016}
twice, we have:
\begin{align*}
\int_{U_{m_1}}\rho'(xB_1^{\mathsf{T}})\,\frac{dx}{\dd(x)^{m_2}}
=\frac{\prod_{j=n-m_1+1}^n(j\fV_j)}{\prod_{j=1}^{m_1}(j\fV_j)}
\int_{(\R^{m_1}\times\{\bn\})^{m_1}}\rho'(xB_1^{\mathsf{T}})\,\dd(x)^{n-m_1-m_2}\,dx
\\
=\frac{\prod_{j=n-m_1+1}^n(j\fV_j)}{\prod_{j=n-m_1-m_2+1}^{n-m_2}(j\fV_j)}
\int_{(\R^{n-m_2}\times\{\bn\})^{m_1}}\rho'(xB_1^{\mathsf{T}})\,dx.
\end{align*}
Hence we obtain the equality in \eqref{MAINTERMLEM5res}.
\end{proof}

Regarding the product of unit ball volumes appearing in 
Lemma \ref{MAINTERMLEM5}, we note that
\begin{align*}
\frac{\prod_{j=n-m_1+1}^n (j\fV_j)}{\prod_{j=n-m_1-m_2+1}^{n-m_2} (j\fV_j)}
\asymp_{m} n^{-m_1m_2/2}\ll 1.
\end{align*}
(We can use this wasteful bound since it will turn out below that our integrals
are exponentially decreasing with respect to $n$.)
Hence it follows from \eqref{Wbetatrivbounds2} and Lemmas \ref{MAINTERMLEM4} and \ref{MAINTERMLEM5}
that 
\begin{align*}
\sum_{\beta\in M_{m_1,m_2}\!(\Z)} W(\beta)\int_{S(\beta)}\rho\bigl(xB_1^{\mathsf{T}},yB_2^{\mathsf{T}}\bigr)\,d\eta_\beta(x,y)
\hspace{150pt}
\\
\ll_k n^{k_1k_2}\cdot %
\int_{(\R^{n-m_2}\times\{\bn\})^{m_1}}\rho'(xB_1^{\mathsf{T}})\,dx
\int_{(\R^{n-m_1}\times\{\bn\})^{m_2}}\trho'(xB_2^{\mathsf{T}})\,dx.
\end{align*}
Note that the subspace 
$(\R^{n-m_2}\times\{\bn\})^{m_1}$
is mapped into
$(\R^{n-m_2}\times\{\bn\})^{k_1}$
by right multiplication by $B_1^{\mathsf{T}}$.
Now let $\sigma'$ be 
the function on $(\R^{n-m_2})^{k_1}$
given by the restriction of %
$\rho'$ to 
$(\R^{n-m_2}\times\{\bn\})^{k_1}$ composed with the obvious
isomorphism $(\R^{n-m_2}\times\{\bn\})^{k_1}=(\R^{n-m_2})^{k_1}$.
We also let $\tsigma'$ be the corresponding function on $(\R^{n-m_1})^{k_2}$
obtained from $\trho'$.
Then the last bound can be written:
\begin{align*}
n^{k_1k_2}\cdot %
\int_{(\R^{n-m_2})^{m_1}}\sigma'(xB_1^{\mathsf{T}})\,dx
\int_{(\R^{n-m_1})^{m_2}}\tsigma'(xB_2^{\mathsf{T}})\,dx.
\end{align*}

Using the above bound,
we conclude that
the total contribution from
all terms in \eqref{MAINTHEOREMabs2resREP}
with $B_1\notin\scrM_{k_1,m_1}'$ or $B_2\notin\scrM_{k_2,m_2}'$
is
\begin{align}\notag
\ll_k n^{k_1k_2}\sum_{m_1=1}^{k_1}\sum_{m_2=1}^{k_2}
\sum_{\langle B_1,B_2\rangle \in (A_{k_1,m_1}\times A_{k_2,m_2})\smallsetminus(\scrM'_{k_1,m_1}\times \scrM'_{k_2,m_2})}
\int_{(\R^{n-m_2})^{m_1}}\sigma'(xB_1^{\mathsf{T}})\,dx
\hspace{50pt}
\\\label{MAINTERM15}
\times\int_{(\R^{n-m_1})^{m_2}}\tsigma'(xB_2^{\mathsf{T}})\,dx.
\end{align}
We will now translate this sum into the notation used in
Rogers' original formulation of his mean value formula
(see Theorem \ref{ROGERSFORMULATHM}).
Let us denote by $\fA(k;m)$ the set of 
matrices of fixed size $m\times k$ appearing in the 
sum in \eqref{ROGERSFORMULATHMres}.\footnote{This set was called $\fA(m)$ in 
Section \ref{ROGERSFORMULAsec}, since there we were working with a single, fixed $k$.}
From the proof in Section \ref{ROGERSFORMULAsec}
of the fact that Theorem \ref{ROGERSFORMULATHMp} is a reformulation of Theorem \ref{ROGERSFORMULATHM},
we recall the definition of the map
$\gamma_2':\fA(k;m)\to M_{k,m}(\Z)^*$,
and the fact that an admissible choice of 
the set of representatives $A_{k,m}$ (cf.\ p.\ \pageref{Akmspec})
is given by
\begin{align*}
A_{k,m}:=\{\gamma_2'(D)\col D\in\fA(k;m)\}.
\end{align*}
Let us set
\begin{align*}
\tscrM_{k,m}':=\{B\transp\col B\in\scrM_{k,m}'\},
\end{align*}
and note that $\tscrM_{k,m}'\subset\fA(k;m)$;
in fact the matrices in
$\tscrM_{k,m}'$ are exactly the matrices considered in
\cite[Thm.\ 3]{aS2011o},
as we noted near the end of Section \ref{MAINTERMsec}.
We claim that the map $\gamma_2'$ can be chosen so that
\begin{align}\label{MAINTERM14}
\gamma_2'(D)=D\transp,\qquad\forall D\in\tscrM_{k,m}'.
\end{align}
Indeed, given $D\in\tscrM_{k,m}'$ one easily verifies that 
there exists a matrix 
$\tau=(\tau_{i,j})\in\GL_k(\Z)$
satisfying $D \tau=\bigl(I_m\: 0\bigr)$
and $\tau_{i,j}=\delta_{i,\nu_j(D\transp)}$ for all $i\in\{1,\ldots,k\}$ and $j\in\{1,\ldots,m\}$.
Hence the relation \eqref{ELEMDIVBASIClempf1} holds
for our matrix $D$
(which has elementary divisors $\ve_1=\cdots=\ve_m=1$)
with $\gamma_1=I$ in $\GL_m(\Z)$
and $\gamma_2=\tau^{-1}$ in $\GL_k(\Z)$.
Thus: We may choose $\gamma_1(D)=I$ 
and $\gamma_2(D)=\tau^{-1}$.
Now by definition, 
$\gamma_2'(D)=\bigl((I_m\: 0)\gamma_2\bigr)\transp$,
and it follows from $D \tau=\bigl(I_m\: 0\bigr)$ that
$\bigl(I_m\: 0\bigr)\gamma_2=D$;
therefore \eqref{MAINTERM14} holds.

From now on we assume that for any $1\leq m\leq k$,
the map $\gamma_2'$ has been chosen so that \eqref{MAINTERM14} holds.
It then follows from the above discussion,
combined with the identity in \eqref{BtoDkeyidentity},
that our bound in \eqref{MAINTERM15}
can be equivalently expressed as
\begin{align}\label{MAINTERM15equ}
 n^{k_1k_2}\sum_{m_1=1}^{k_1}\sum_{m_2=1}^{k_2}
\sum_{\langle D_1,D_2\rangle \in (\fA(k_1,m_1)\times \fA(k_2,m_2))\smallsetminus(\tscrM'_{k_1,m_1}\times \tscrM'_{k_2,m_2})}
J_{n-m_2}(\sigma';D_1)\,J_{n-m_1}(\tsigma';D_2),
\end{align}
where for any $1\leq m\leq k<\tn$, $D\in\fA(k,m)$ and (measurable) function
$\sigma:(\R^{\tn})^m\to\R_{\geq0}$, $J_{\tn}(\sigma;D)$ is defined by
\begin{align*}
J_{\tn}(\sigma;D):=\Big(\frac{e_1}{q}\cdots\frac{e_m}{q}\Big)^{\tn}
\int_{(\R^{\tn})^m}\sigma(xq^{-1}D)\,dx,
\end{align*}
where $q=q(D)$ is as in Theorem \ref{ROGERSFORMULATHM},
and $e_1,\ldots,e_m$ are the elementary divisors of $D$.
Note here that for any $1\leq m_1\leq k_1$ and $1\leq m_2\leq k_2$,
the sum over $\langle D_1,D_2\rangle$ appearing in \eqref{MAINTERM15equ}
can be decomposed as the sum of
the three products
\begin{align}\label{MAINTERM17}
\Biggl(\sum_{D_1\in\fA(k_1,m_1)\smallsetminus\tscrM'_{k_1,m_1}}
J_{n-m_2}(\sigma';D_1)\Biggr)
\Biggl(\sum_{D_2\in\fA(k_2,m_2)\smallsetminus\tscrM'_{k_2,m_2}}
J_{n-m_1}(\tsigma';D_2)\Biggr)
\end{align}
and
\begin{align}\label{MAINTERM18}
\Biggl(\sum_{D_1\in\tscrM'_{k_1,m_1}}
J_{n-m_2}(\sigma';D_1)\Biggr)
\Biggl(\sum_{D_2\in\fA(k_2,m_2)\smallsetminus\tscrM'_{k_2,m_2}}
J_{n-m_1}(\tsigma';D_2)\Biggr)
\end{align}
and
\begin{align}\label{MAINTERM19}
\Biggl(\sum_{D_1\in\fA(k_1,m_1)\smallsetminus\tscrM'_{k_1,m_1}}
J_{n-m_2}(\sigma';D_1)\Biggr)
\Biggl(\sum_{D_2\in\tscrM'_{k_2,m_2}}
J_{n-m_1}(\tsigma';D_2)\Biggr).
\end{align}

Now the sum
\begin{align}\label{MAINTERM20}
\sum_{D_1\in\fA(k_1,m_1)\smallsetminus\tscrM'_{k_1,m_1}}
J_{n-m_2}(\sigma';D_1)
\end{align}
which appears as a factor in \eqref{MAINTERM17} and \eqref{MAINTERM19},
is exactly the sum which is treated in
\cite[pp.\ 950(top)--951(middle)]{aS2011o}
(applied with ``$n$'' taken to be $n-m_2$ in our notation).
It should here be noted that
(cf.\ \eqref{rhoptrhopDEF})
\begin{align*}
\sigma'(\vecv_1,\ldots,\vecv_{k_1})
:=\prod_{j=1}^{k_1}\sigma_j(\vecv_j)
\end{align*}
where for each $j$, %
$\sigma_j$ is the characteristic function of the open $(n-m_2)$-ball of 
radius $(V_j/\fV_n)^{1/n}$ centered at the origin, with the origin removed.
The volume of the ball which is the support of $\sigma_j$ is therefore
\begin{align}\label{MAINTERM22}
\fV_{n-m_2}(V_j/\fV_n)^{(n-m_2)/n}\ll_m V_j
\end{align}
(in fact the left-hand side tends to $e^{m_2/2}V_j$ as $n\to\infty$).
The key points for us is that these volumes stay \textit{bounded} as $n\to\infty$,
since the volumes $V_1,\ldots,V_{k_1}$ are fixed.
Therefore, it follows from the bounds in 
\cite[pp.\ 950(top)--951(middle)]{aS2011o}
that 
the sum in \eqref{MAINTERM20}
is $\ll_m (3/4)^{\frac{n-m_2}2}\ll_m(3/4)^{\frac n2}$
(the implied constant also depends on the $V_j$s).

In the same way, we also have
\begin{align}\label{MAINTERM21}
\sum_{D_2\in\tscrM'_{k_2,m_2}}
J_{n-m_1}(\tsigma';D_2)
\ll_m(3/4)^{\frac n2}.
\end{align}
Furthermore, by the proof of
\cite[Prop.\ 1]{aS2011o},
for each $D_1$ in the finite set
$\tscrM'_{k_1,m_1}$ 
we have that $J_{n-m_2}(\sigma';D_1)$ tends to a finite limit as $n\to\infty$;
the same thing of course also holds for
$J_{n-m_1}(\tsigma';D_2)$ for each fixed $D_2\in\tscrM'_{k_2,m_2}$.

In view of these bounds, we conclude that the expression in
\eqref{MAINTERM15equ} tends to zero
exponentially rapidly
as $n\to\infty$.
This completes the proof of Theorem \ref{MOMENTTHM}.
\hfill$\square$ $\square$ $\square$

\subsection{Proof of Theorem \ref*{JOINTPOISSONTHEOREM}}
\label{finalproofSEC}

As in the statement of
Theorem \ref{JOINTPOISSONTHEOREM},
let $0<T_1<T_2<T_3<\cdots$
and $0<T_1'<T_2'<T_3'<\cdots$ denote the points of two independent 
Poisson processes on $\R^+$ with constant intensity $\frac{1}{2}$.
For any fixed $V>0$ we introduce the two integer-valued random variables
\begin{align*}
N(V):=\#\{j\col T_j<V\}
\qquad\text{and}\qquad
\tN(V):=\#\{j\col T_j'<V\}.
\end{align*}
Now by standard arguments
(cf.\ the proof of 
\cite[Cor.\ 1]{aS2011o}),
Theorem \ref{MOMENTTHM}
implies that,
for any fixed 
$k_1,k_2\in\Z_{\geq0}$ and any fixed real numbers
$0<V_1\leq V_2\leq\cdots\leq V_{k_1}$
and $0<W_1\leq W_2\leq\cdots\leq W_{k_2}$,
the random vector
\begin{align*}
\Bigl(\tfrac12 N_1(L),\ldots,\tfrac12 N_{k_1}(L),\tfrac12 \tN_1(L^*),\ldots,\tfrac12 \tN_{k_2}(L^*)\Bigr)
\end{align*}
converges in distribution to the random vector
\begin{align*}
\Bigl(N(V_1),\ldots,N(V_{k_1}),\tN(W_1),\ldots,\tN(W_{k_2})\Bigr).
\end{align*}
Theorem \ref{JOINTPOISSONTHEOREM} is an immediate consequence of this fact.
\hfill$\square$ $\square$ $\square$

\vspace{20pt}

\end{document}